\documentclass[11pt]{article}
\usepackage[utf8]{inputenc}
\usepackage[margin=1in]{geometry}

\usepackage{amsmath, amsthm, thmtools}
\newcommand{\proofbox}{\blacksquare}

\usepackage{charter}
\DeclareMathAlphabet{\mathcal}{OMS}{zplm}{m}{n}

\usepackage{mathtools, amsfonts, amssymb, mathrsfs}
\usepackage{enumerate,enumitem}
\usepackage{subcaption}
\usepackage{color,soul}
\usepackage{algorithm, algpseudocode, algorithmicx}
\usepackage{comment}
\usepackage{xspace}
\usepackage{mleftright}

\newcommand{\real}{\mathbb{R}}

\DeclareMathOperator{\diag}{diag}

\newcommand{\mat}[1]{\boldsymbol{#1}}
\renewcommand{\vec}[1]{\boldsymbol{#1}}

\newcommand{\norm}[1]{\mleft\| #1 \mright\|}

\newcommand{\expmat}[1]{\begin{bmatrix} #1 \end{bmatrix}}
\newcommand{\twobytwo}[4]{\expmat{#1 & #2 \\ #3 & #4}}
\newcommand{\twobyone}[2]{\expmat{#1 \\ #2}}
\newcommand{\onebytwo}[2]{\expmat{#1 & #2}}

\newcommand{\Id}{\mathbf{I}}

\newcommand{\order}{\mathcal{O}}

\DeclareMathOperator*{\argmin}{argmin}
\newcommand{\set}[1]{\mathsf{#1}}

\renewcommand{\hat}[1]{\widehat{#1}}

\newcommand{\actionbox}[1]{\begin{center} \vspace{0.5pc}
\fbox{ \begin{minipage}{0.9\textwidth}
\begin{center}#1\end{center}
\end{minipage}
} \vspace{0.5pc}
\end{center}}

\usepackage{constants}
\renewconstantfamily{normal}{symbol=c}
\newcommand{\dta}[1]{\mat{\Delta}{\mat{#1}}}

\DeclareMathOperator{\range}{range}
\newcommand{\QR}{\textsf{QR}\xspace}
\newcommand{\LU}{\textsf{LU}\xspace}
\newcommand{\err}{\mathrm{err}}
\newcommand{\rerr}{\mathrm{rerr}}
\usepackage{graphicx}

\newif\ifsiam

\newcommand{\isiam}[1]{\ifsiam #1 \fi}
\newcommand{\iarxiv}[1]{\ifsiam \else #1 \fi}
\newcommand{\siamorarxiv}[2]{\ifsiam #1 \else #2 \fi}

\siamfalse

\makeatletter
\def\th@plain{%
  \thm@notefont{}
  \itshape 
}
\def\th@definition{%
  \thm@notefont{}
  \normalfont 
}
\makeatother

\usepackage[pagebackref]{hyperref}
\usepackage[nameinlink,capitalise]{cleveref}
\usepackage{doi}

\crefname{equation}{}{}
\crefname{section}{section}{sections}
\crefname{appendix}{appendix}{appendices}
\newcommand*{\email}[1]{\href{mailto:#1}{\nolinkurl{#1}} } 

\declaretheorem[name=Theorem]{theorem}
\declaretheorem[name=Informal Theorem]{inftheorem}
\declaretheorem[name=Proposition,numberlike=theorem]{proposition}
\declaretheorem[name=Fact,numberlike=theorem]{fact}

\declaretheorem[name=Lemma,numberlike=theorem]{lemma}

\theoremstyle{remark}

\theoremstyle{definition}
\declaretheorem[name=Definition,numberlike=theorem]{definition}
\usepackage{mdframed}
\newmdtheoremenv{question}{Question}

\numberwithin{equation}{section}

\crefname{fact}{Fact}{Facts}
\Crefname{fact}{Fact}{Facts}

\graphicspath{{figs/}}

\title{Fast and Forward Stable Randomized Algorithms \\ for Linear Least-Squares Problems\thanks{This work was supported by NSF FRG 1952777, the Carver Mead New Horizons Fund, and the U.S. Department of Energy, Office of Science, Office of Advanced Scientific Computing Research, Department of Energy Computational Science Graduate Fellowship under Award Number DE-SC0021110.}}
\author{Ethan N. Epperly\thanks{Division of Computing and Mathematical Sciences, California Institute of Technology, Pasadena, CA 91125 USA (\email{eepperly@caltech.edu}, \url{https://ethanepperly.com})}}
\date{\today}

\begin{document}

\maketitle

\begin{abstract}
    Iterative sketching and sketch-and-precondition are randomized algorithms used for solving overdetermined linear least-squares problems.
When implemented in exact arithmetic, these algorithms produce high-accuracy solutions to least-squares problems faster than standard direct methods based on \QR factorization.
Recently, Meier, Nakatsukasa, Townsend, and Webb demonstrated numerical instabilities in a version of sketch-and-precondition in floating point arithmetic (\href{https://arxiv.org/abs/2302.07202}{arXiv:2302.07202}).
The work of Meier et al.\ raises the question: Is there a randomized least-squares solver that is both fast and stable?
This paper resolves this question in the affirmative by proving that iterative sketching, appropriately implemented, is forward stable.
Numerical experiments confirm the theoretical findings, demonstrating that iterative sketching is stable and faster than \QR-based solvers for large problem instances.

\end{abstract}

\section{Introduction} \label{sec:introduction}

The overdetermined linear least-squares problem \cite{Bjo96}
\begin{equation} \label{eq:ls}
  \vec{x} = \argmin_{\vec{y} \in \real^n} \norm{\vec{b} - \mat{A}\vec{y}} \quad \text{for $\mat{A}\in\real^{m\times n}$, $\vec{b} \in \real^m$, $m\ge n$}
\end{equation}
is one of the core problems in computational mathematics.
Here, $\norm{\cdot}$ denotes the Euclidean norm.
The textbook algorithm for this problem is by Householder \QR factorization \cite[\S5.2.2]{GV13}, which runs in $\order(mn^2)$ operations.

\begin{algorithm}[t]
  \caption{Sketch-and-precondition \label{alg:spre}}
  \begin{algorithmic}[1]
    \Require Matrix $\mat{A} \in \real^{m\times n}$, right-hand side $\vec{b} \in \real^m$, embedding dimension $d$, iteration count $q$
    \Ensure Approximation solution $\vec{x}_q$ to least-squares problem
    \State $\mat{S} \gets \text{$d\times m$ subspace embedding}$
    \State $\mat{B} \gets \mat{S}\mat{A}$
    \State $(\mat{Q},\mat{R}) \gets \Call{QR}{\mat{B}, \texttt{'econ'}}$
    \State $\vec{x}_0 \gets \mat{R}^{-1}(\mat{Q}^\top (\mat{S}\vec{b}))$ \Comment{Sketch-and-solve initialization (\cref{sec:sketch-and-solve})}
    \State $\vec{x}_q \gets \Call{LSQR}{\mat{A},\vec{b},\texttt{initial}=\vec{x}_0,\texttt{iterations}=q,\texttt{preconditioner}=\mat{R}}$
  \end{algorithmic}
\end{algorithm}

In the past two decades, researchers in the field of \emph{randomized numerical linear algebra} \cite{MT20a,MDM+22} have developed least-squares solvers that are faster than Householder \QR.
The core ingredient for these methods is a \emph{fast random embedding} \cite[\S9]{MT20a}, informally defined to be a random matrix $\mat{S} \in \real^{d\times m}$ satisfying the following properties:
\begin{enumerate}
\item With high probability, $\norm{\mat{S}\vec{v}} \approx \norm{\vec{v}}$ for all vectors $\vec{v} \in \range(\onebytwo{\mat{A}}{\vec{b}})$.
\item The output dimension $d$ is a small multiple of $n$ or $n \log n$.
\item Matrix--vector products $\vec{v} \mapsto \mat{S}\vec{v}$ can be computed in $\order(m\log m)$ operations.
\end{enumerate}
\Cref{sec:subspace} provides a precise definition and constructions of fast random embeddings.
Random embeddings are a core ingredient in randomized least-squares solvers:
\begin{itemize}
\item \textbf{Sketch-and-precondition \cite{RT08,AMT10}.} Apply an embedding $\mat{S}$ to $\mat{A}$ and compute an (economy-size) \QR decomposition 
\begin{equation} \label{eq:sketch-and-precondition-QR}
  \mat{S}\mat{A} = \mat{Q}\mat{R}.
\end{equation}
Use the matrix $\mat{R}$ as a right-preconditioner to solve \cref{eq:ls} by the LSQR iterative method \cite{PS82}, a Krylov method for least-squares. Details are provided in \cref{alg:spre}.
\item \textbf{Iterative sketching \cite{PW16a,OPA19,LP21}.}
Starting from an initial solution $\vec{x}_0\in\real^n$, generate iterates $\vec{x}_1,\vec{x}_2,\ldots$ by solving a sequence of minimization problems
\begin{equation} \label{eq:it-sk-seq-min}
  \vec{x}_{i+1} = \argmin_{\vec{y} \in \real^n} \norm{(\mat{S}\mat{A})\vec{y}}^2 - \vec{y}^\top \mat{A}^\top (\vec{b} - \mat{A}\vec{x}_i) \quad \text{for } i = 0,1,2,\ldots.
\end{equation}
Note that this version of iterative sketching uses a single random embedding $\mat{S}$ for all iterations. See \cref{alg:it-sk} for a stable implementation developed in the present work.
\end{itemize}
When implemented in exact arithmetic and with the appropriate initialization $\vec{x}_0$ (\cref{eq:sketched_ls} below), both sketch-and-precondition and iterative sketching produce a solution $\vec{x}_q$ which satisfies
\begin{equation*}
  \norm{\vec{r}(\vec{x}_q)} \le (1+\delta) \cdot \norm{\vec{r}(\vec{x})},
\end{equation*}
where $\delta > 0$ is an accuracy parameter that can be made smaller by running for more iterations.
Here, and going forward, $\vec{r}(\vec{y}) \coloneqq \vec{b} - \mat{A}\vec{y}$ denotes the residual for a vector $\vec{y}$.
With appropriate choice of $\mat{S}$, both sketch-and-precondition and iterative sketching have \emph{the same} asymptotic cost of
\begin{equation*}
  q = \order(\log(1/\delta)) \text{ iterations} \quad \text{and} \quad \order(mn \log (n/\delta) + n^3 \log n) \text{ operations}.
\end{equation*}
In principle, these algorithms allow us to solve least-squares problems faster than Householder \QR.

In recent work \cite{MNTW23}, Meier, Nakatsukasa, Townsend, and Webb studied the \emph{numerical stability} of randomized least-squares solvers.
Their conclusions were as follows:
\begin{itemize}
    \item Sketch-and-precondition with the zero initialization $\vec{x}_0 = \vec{0}$ is \emph{numerically unstable}.
    When implemented in floating point arithmetic, it produces errors $\norm{\vec{x} - \vec{\hat{x}}}$ and residuals $\norm{\vec{r}(\vec{\hat{x}})}$ much larger than Householder \QR factorization.
    \item Empirically, the sketch-and-solve initialization (see \cref{eq:sketched_ls} below) significantly improves the stability of sketch-and-precondition, but the method still fails to be backward stable.
    \item Meier et al.\ introduce sketch-and-apply, a slower, backward stable version of sketch-and-precondition that runs in $\order(mn^2)$ operations, the same as Householder \QR.
\end{itemize}
For a randomized least-squares solver to see wide deployment in general-purpose software, it is desirable to have an algorithm that is both fast and stable.
The work of Meier et al.\ raises doubts about whether existing randomized least-squares solvers meet both criteria, leading to the question:
\begin{question} \label{que:main}
\vspace{-1.2ex}Is there a randomized least-squares algorithm that is both (asymptotically) faster than Householder \QR and numerically stable?
\end{question}
In this paper, we answer \cref{que:main} in the affirmative, showing that a careful implementation of \textbf{iterative sketching} achieves both criteria.
\begin{algorithm}[t]
  \caption{Iterative sketching \label{alg:it-sk}}
  \begin{algorithmic}[1]
    \Require Matrix $\mat{A} \in \real^{m\times n}$, right-hand side $\vec{b} \in \real^m$, embedding dimension $d$, iteration count $q$
    \Ensure Approximation solution $\vec{x}_q$ to least-squares problem
    \State $\mat{S} \gets \text{$d\times m$ subspace embedding}$
    \State $\mat{B} \gets \mat{S}\mat{A}$
    \State $(\mat{Q},\mat{R}) \gets \Call{QR}{\mat{B}, \texttt{'econ'}}$
    \State $\vec{x}_0 \gets \mat{R}^{-1}(\mat{Q}^\top (\mat{S}\vec{b}))$\Comment{Sketch-and-solve initialization (\cref{sec:sketch-and-solve})}
    \For{$i = 0,1,2,\ldots,q-1$}
    \State $\vec{r}_i \gets \vec{b} - \mat{A}\vec{x}_i$
    \State $\vec{c}_i \gets \mat{A}^\top \vec{r}_i$ 
    \State $\vec{d}_i \gets \mat{R}^{-1}(\mat{R}^{-\top}\vec{c}_i)$ \Comment{Two triangular solves}
    \State $\vec{x}_{i+1} \gets \vec{x}_i + \vec{d}_i$ \label{line:update}
    \EndFor
  \end{algorithmic}
\end{algorithm}
Our main contributions are as follows:
\begin{enumerate}
\item \textbf{Stable implementation.} We provide a stable implementation of iterative sketching in \cref{alg:it-sk}.
  We demonstrate the importance of using our stable implementation in \cref{fig:bad}, where we show two other plausible implementations that fail to be numerically stable.
\item \textbf{Proof of forward stability.} We prove that iterative sketching is \emph{forward stable}: 
  \begin{inftheorem}[Iterative sketching is forward stable] \label{thm:it-sk-informal}
    In floating point arithmetic, \cref{alg:it-sk} produces iterates $\vec{\hat{x}}_0,\vec{\hat{x}}_1,\ldots$ whose error $\norm{\vec{x} - \vec{\hat{x}}_i}$ and residual error $\norm{\vec{r}(\vec{x}) - \vec{r}(\vec{\hat{x}}_i)}$ converge geometrically until they reach roughly the same accuracy as Householder \QR.
  \end{inftheorem}
  See \cref{thm:it_sk_forward} for a precise statement.
  The forward stability of iterative sketching is a weaker guarantee than the backward stability (see \cref{eq:ls_backward}) of Householder \QR.
  Despite this, a forward-stable least-squares solver is adequate for many applications.
  In particular, the residual norms $\norm{\vec{r}(\vec{\hat{x}})}$ for the solutions $\vec{\hat{x}}$ produced by iterative sketching and Householder \QR are comparably small.
  Our stability analysis for iterative sketching is modelled after the Bj\"ork's analysis of the delicate numerical properties the \emph{corrected seminormal equation} method \cite{Bjo87}.
\end{enumerate}
This work is the first to show a randomized least-squares solver is both numerically stable, in any sense, and (asymptotically) faster than Householder \QR, affirmatively resolving \cref{que:main}.

\subsection{Numerical evidence} \label{sec:numerical}

In this section, we demonstrate the forward stability of iterative sketching numerically.

\subsubsection{Experimental setup}
\label{sec:exp-setup}

We adopt a similar setup to \cite[Fig.~1.1]{MNTW23}.
Set $m\coloneqq 4000$ and $n \coloneqq 50$, and choose parameters $\kappa \ge 1$ for the condition number of $\mat{A}$ and $\beta \ge 0$ for the residual norm $\norm{\vec{r}(\vec{x})}$.
To generate $\mat{A}$, $\vec{x}$, and $\vec{b}$, do the following:
\begin{itemize}
\item Choose Haar random orthogonal matrices $\mat{U} = \onebytwo{\mat{U}_1}{\mat{U}_2} \in \real^{m\times m}$ and $\mat{V} \in \real^{n\times n}$, and partition $\mat{U}$ so that $\mat{U}_1 \in\real^{m\times n}$.
\item Set $\mat{A}\coloneqq\mat{U}_1\mat{\Sigma}\mat{V}^\top$ where $\mat{\Sigma}$ is a diagonal matrix with logarithmically equispaced entries between $1$ and $1/\kappa$.
\item Form vectors $\vec{w} \in \real^n,\vec{z} \in\real^{m-n}$ with independent standard Gaussian entries.
\item Define the solution $\vec{x} \coloneqq \vec{w} / \norm{\vec{w}}$, residual $\vec{r}(\vec{x}) = \beta \cdot \mat{U}_2 \vec{z} / \norm{\mat{U}_2 \vec{z}}$, and right-hand side $\vec{b} \coloneqq \mat{A}\vec{x} + \vec{r}(\vec{x})$.
\end{itemize}
For all randomized methods, we use a sparse sign embedding $\mat{S} \in \real^{d \times n}$ with sparsity $\zeta \coloneqq 8$ and embedding dimension $d \coloneqq 20n$ (see \cref{sec:subspace}).
All experiments are performed in double precision ($u\approx 10^{-16}$) in MATLAB 2023b on a Macbook Pro with Apple M3 Pro and 18 GB of memory.
Code for all experiments can be found at:
\actionbox{\url{https://github.com/eepperly/Iterative-Sketching-Is-Stable}}

\subsubsection{Error metrics}

What does it mean to solve a least-squares problem accurately?
Unfortunately, there is not one single error metric for least-squares problems that is appropriate in all contexts.
In fact, there are three useful error metrics:
\begin{itemize}
    \item \textbf{Forward error.} Perhaps the most natural error metric is the (relative) forward error
    \begin{equation*}
        \operatorname{FE}(\vec{\hat{x}}) \coloneqq \frac{\norm{\vec{x}-\vec{\hat{x}}}}{\norm{\vec{x}}}.
    \end{equation*}
    The forward error quantifies how close the computed solution $\vec{\hat{x}}$ is to the true solution $\vec{x}$.
    \item \textbf{Residual error.} The (relative) residual error is 
    \begin{equation*}
        \operatorname{RE}(\vec{\hat{x}}) \coloneqq \frac{\norm{\vec{r}(\vec{x})-\vec{r}(\vec{\hat{x}})}}{\norm{\vec{r}(\vec{x})}}.
    \end{equation*}
    The residual error measures the \emph{suboptimality} of $\vec{\hat{x}}$ as a solution to the least-squares minimization problem \cref{eq:ls}, according to the relation:
    \begin{equation*}
        \norm{\vec{r}(\vec{\hat{x}})} = \sqrt{ 1 + \operatorname{RE}(\vec{\hat{x}})^2 } \cdot \norm{\vec{r}(\vec{x})}.
    \end{equation*}
    \item \textbf{Backward error.} The (relative) backward error \cite[\S20.7]{Hig02} is 
    \begin{equation*}
      \operatorname{BE}(\vec{\hat{x}}) \coloneqq \min \left\{ \frac{\norm{\dta{A}}_{\rm F}}{\norm{\mat{A}}_{\rm F}} : \vec{\hat{x}} = \argmin_{\vec{v}} \norm{\vec{b} - (\mat{A}+\dta{A})\vec{v}} \right\}.
    \end{equation*}
    If the backward error is small, then $\vec{\hat{x}}$ is the \emph{true solution to nearly the right} least-squares problem.
    Informally, a method is said to be \emph{backward stable} if $\operatorname{BE}(\vec{\hat{x}})\approx u$, where $u$ is the unit roundoff ($u\approx 10^{-16}$ in double precision).
\end{itemize}
There are applications for which each of these error metrics is the appropriate measure of accuracy.
Wedin's perturbation theorem (\cref{thm:wedin}) bounds the forward and residual errors in terms of the backward error, so achieving a small backward error (i.e., backward stability) is considered the gold standard stability property for a least-squares algorithm.
Informally, a method is called \emph{forward stable} if the forward and residual errors are as small as guaranteed by Wedin's theorem for a backward stable method.

\subsubsection{Iterative sketching is forward stable}
\label{sec:it-sk-init-numerical}

\begin{figure}[t]
  \centering
  
  $\norm{\vec{r}(\vec{x})} = 10^{-12}$
  
  \includegraphics[width=0.45\textwidth]{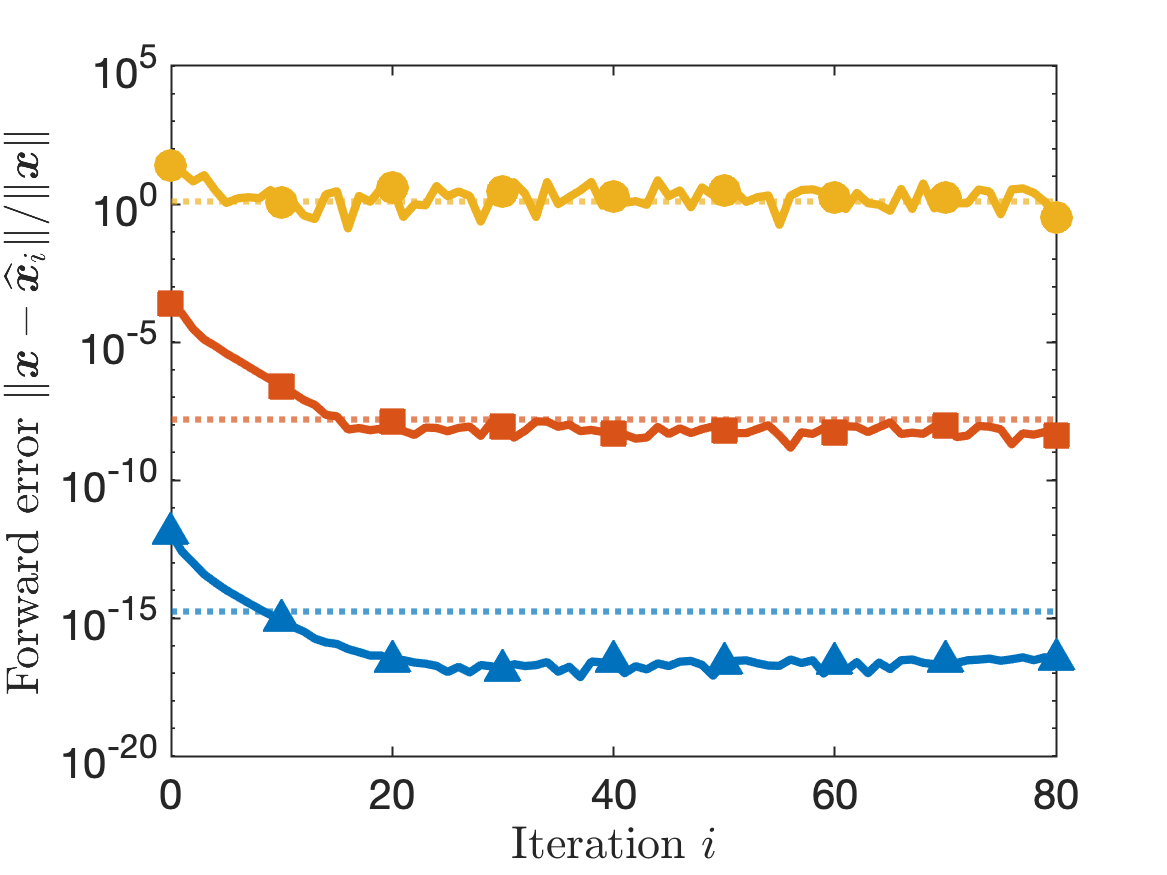}
  \includegraphics[width=0.45\textwidth]{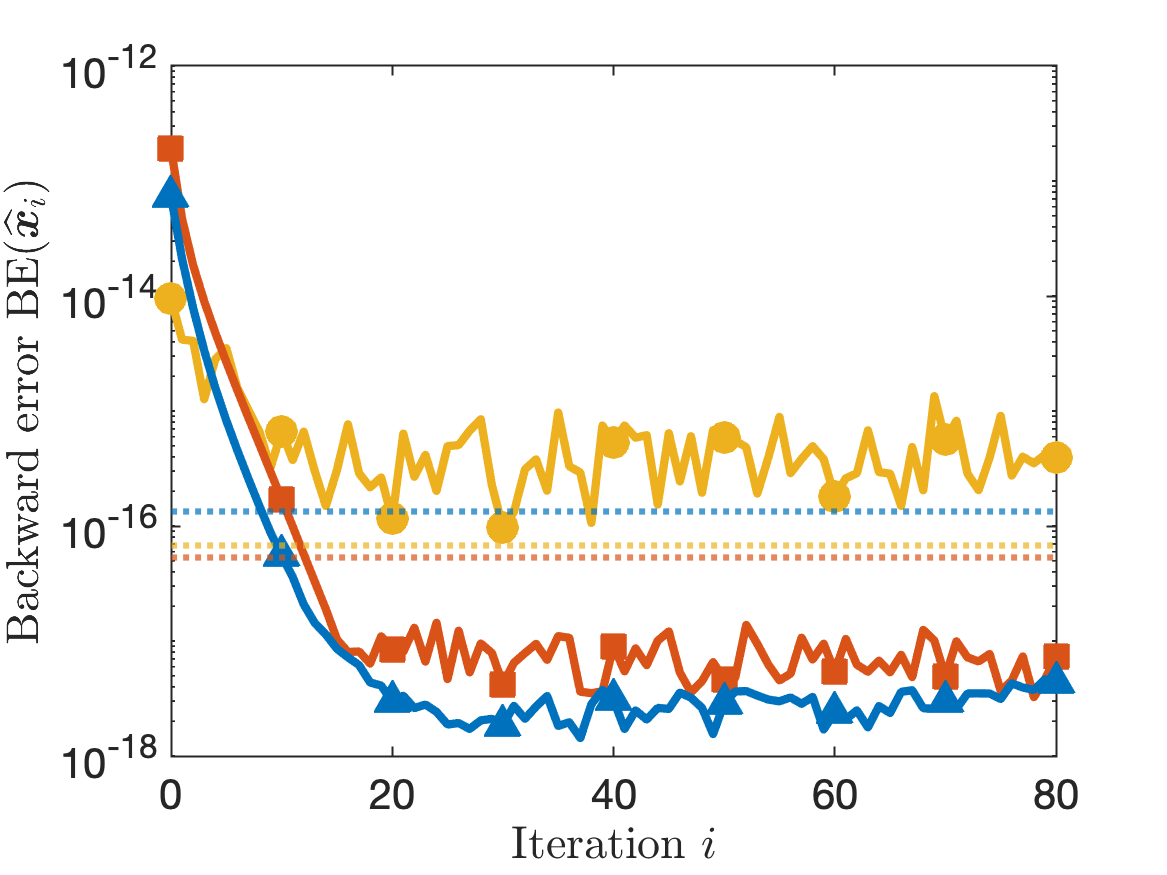} 
  
  $\norm{\vec{r}(\vec{x})} = 10^{-3}$
  
  \includegraphics[width=0.45\textwidth]{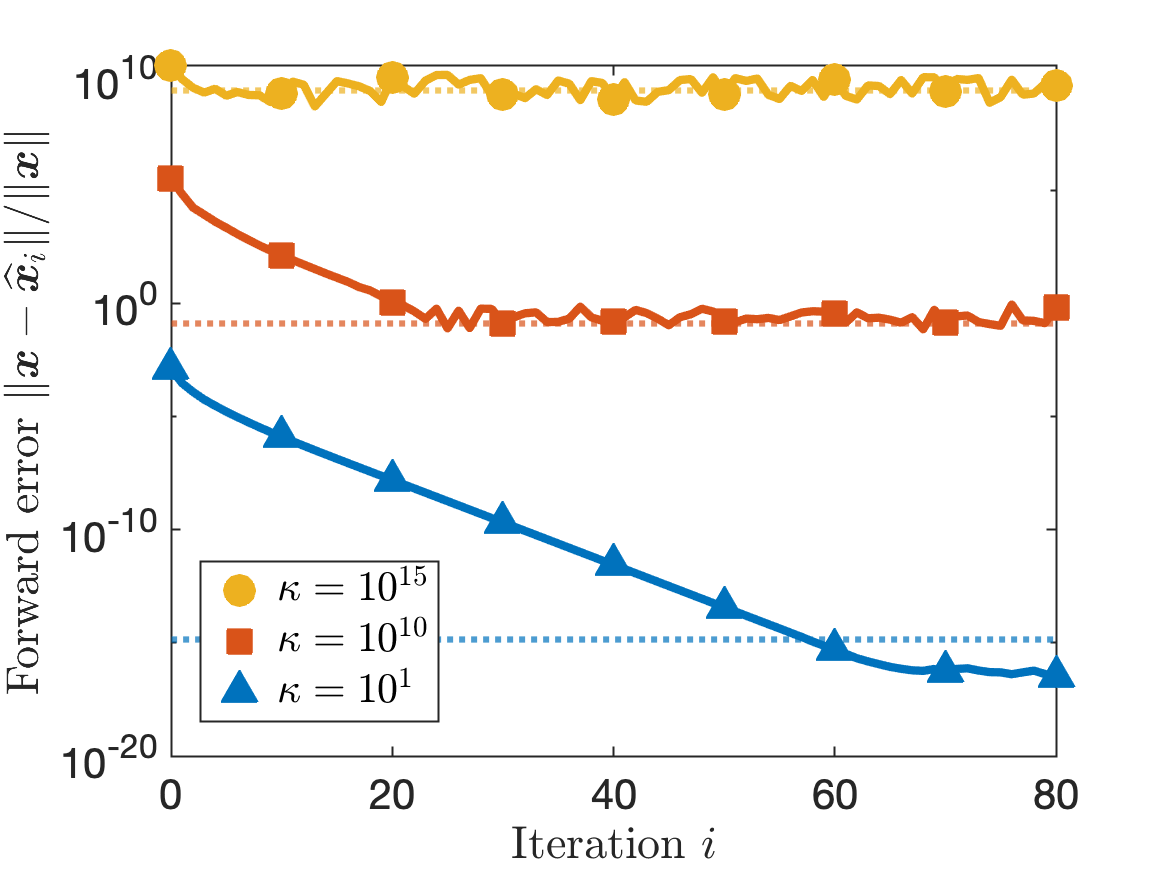}
  \includegraphics[width=0.45\textwidth]{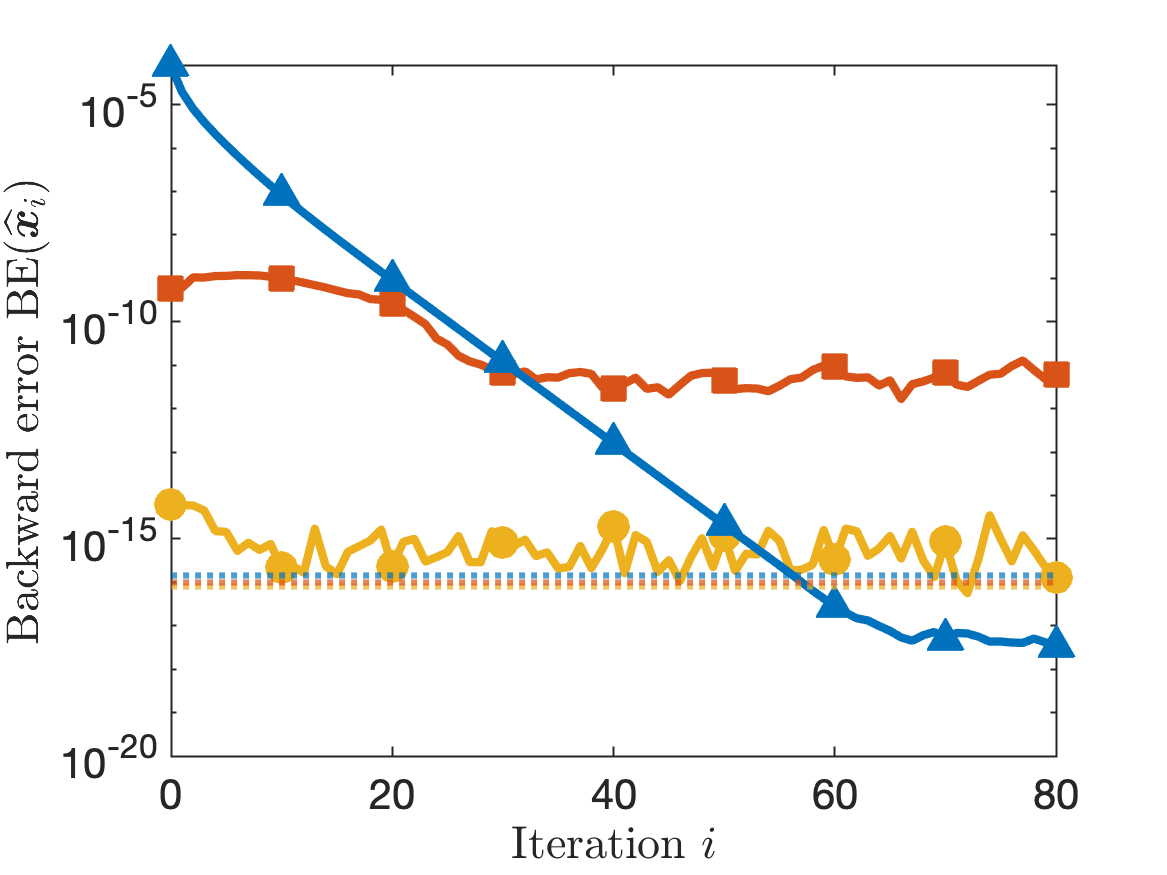} 
  
  \caption{Forward (\emph{left}) and backward (\emph{right}) error for iterative sketching (solid lines).
    We consider three different condition numbers $\kappa = 10^1$ (blue triangles), $\kappa = 10^{10}$ (red squares), and $\kappa = 10^{15}$ (yellow circles) and two residuals $\norm{\vec{r}(\vec{x})} = 10^{-12}$ (\emph{top}) and $\norm{\vec{r}(\vec{x})} = 10^{-3}$ (\emph{bottom}).
    Reference accuracies for Householder \QR are shown as dotted lines.} \label{fig:numerical}
\end{figure}

\Cref{fig:numerical} shows the forward error $\operatorname{FE}(\vec{\hat{x}})$ (\emph{left}) and backward error $\operatorname{BE}(\vec{\hat{x}})$ (\emph{right}) for the computed iterative sketching iterates $\vec{\hat{x}}_0,\vec{\hat{x}}_1,\ldots$ produced by the stable implementation from \cref{alg:it-sk}.
(The residual error as well is shown \siamorarxiv{in the supplementary material, \cref{fig:numerical_full}.)}{below in \cref{fig:numerical_full}.)}
Colors and markers distinguish three values $\kappa = 10^1,10^{10},10^{15}$, and we test two values of $\norm{\vec{r}(\vec{x})}$ (\emph{top} and \emph{bottom}).
For each value of $\kappa$, a references accuracy for Householder \QR is shown as a dashed line.
Our takeaways are as follows:
\begin{itemize}
\item \textbf{Forward stability, geometric convergence.} The forward and residual errors converge geometrically until they reach or fall below the error of Householder \QR, a backward stable method.
  In particular, we confirm the forward stability of iterative sketching (\cref{thm:it_sk_forward}).
\item \textbf{Backward stability?}
  In our results, we achieve a backward error comparable with Householder \QR when the residual is small $\norm{\vec{r}(\vec{x})} = 10^{-12}$.
  However, for large residual $\norm{\vec{r}(\vec{x})} = 10^{-3}$, the backward error for iterative sketching can be much larger than Householder \QR.
  We conclude that iterative sketching is not backward stable.
\end{itemize}

\subsubsection{Comparison to sketch-and-precondition}
\label{sec:other}

\begin{figure}[t]
  \centering
  \includegraphics[width=0.45\textwidth]{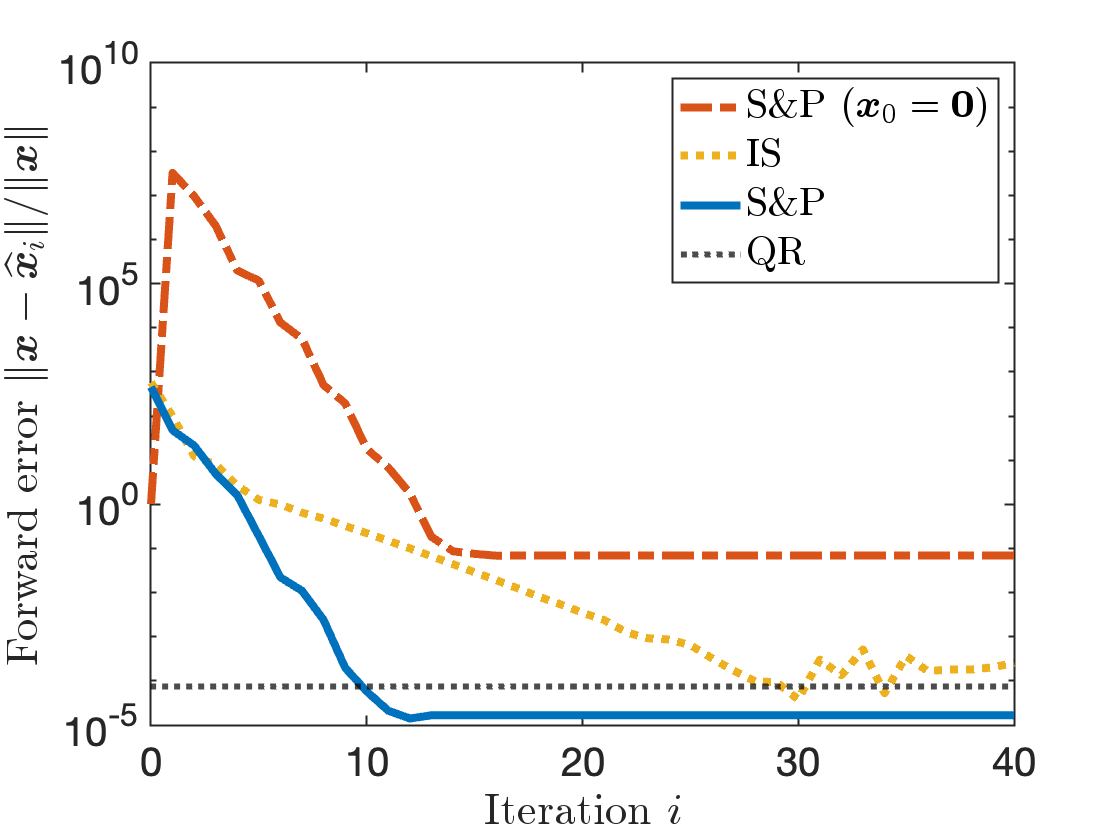}
  \includegraphics[width=0.45\textwidth]{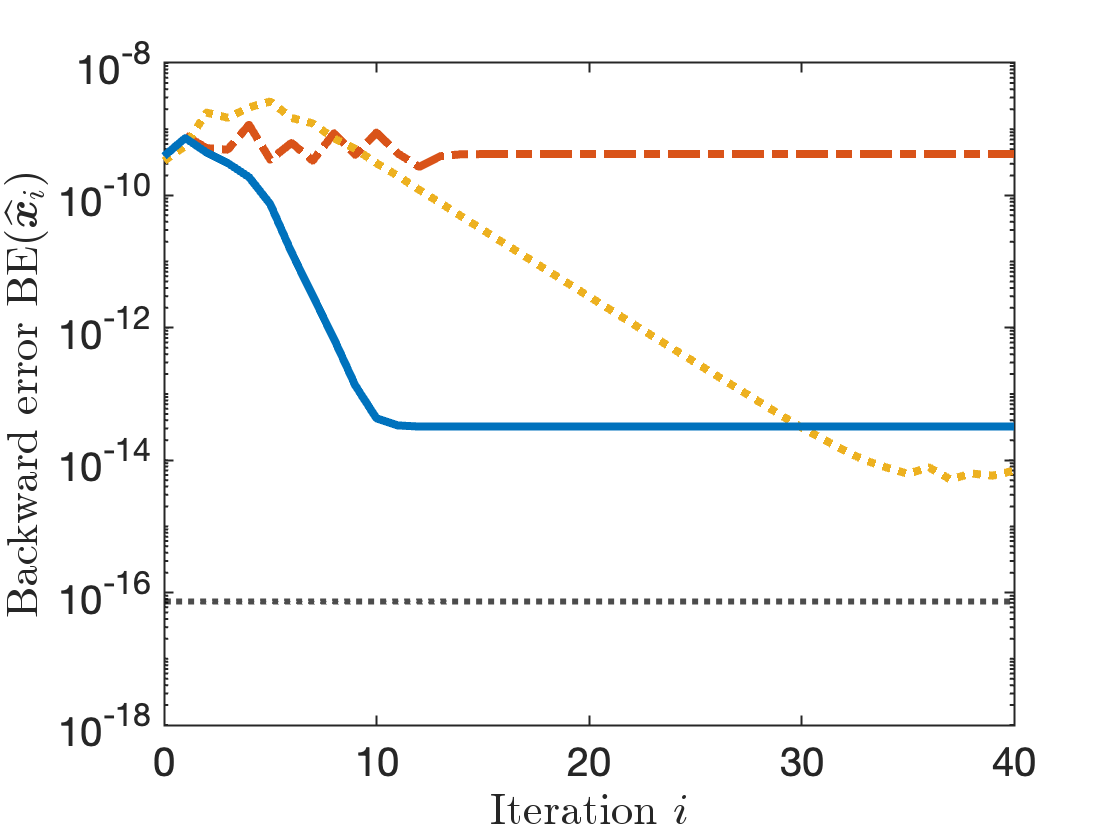} 
    
  \caption{Forward (\emph{left}) and backward (\emph{right}) error for sketch-and-precondition (\textbf{S{\&}P}) and iterative sketching (\textbf{IS}, yellow dotted) with $\kappa = 10^{10}$ and $\norm{\vec{r}(\vec{x})} = 10^{-6}$.
    We consider sketch-and-precondition with both the zero initialization (red dashed) and the sketch-and-solve initialization (blue solid).
    Reference errors for Householder \QR are shown as black dotted lines.} \label{fig:other}
\end{figure}

\Cref{fig:other} compares the accuracy of iterative sketching (\textbf{IS}, \cref{alg:it-sk}) and sketch-and-precondition (\textbf{S\&P}) with both the zero intialization $\vec{x}_0 = \vec{0}$ and sketch-and-solve initialization (\cref{eq:sketched_ls} below).
We use the test case $\kappa = 10^{10}$ and $\norm{\vec{r}(\vec{x})} = 10^{-6}$.
Our conclusions are as follows:
\begin{itemize}
\item \textbf{Stability of sketch-and-precondition?} 
  We confirm the findings of \cite{MNTW23}; sketch-and-precondition is unstable with the zero initialization.
  The stability of sketch-and-precondition is greatly improved using the sketch-and-solve initialization, but the method is still not backward stable.
\item \textbf{Convergence rate.} For our choice of $d=20n$, sketch-and-precondition converges $2.2\times$ faster than iterative sketching. 
  We discuss variants of iterative sketching with improved convergence rates in \cref{sec:damp-mom}.
\end{itemize}

\subsection{Plan for paper} \label{sec:plan}

\Cref{sec:background} discusses background material on perturbation theory and numerical stability for least-squares problem, random subspace emeddings, and the sketch-and-solve method.
\Cref{sec:it-sk} discusses the iterative sketching method including stable implementation, practical implementation guidance, and more numerical results.
\Cref{sec:it-sk-for} presents our proof of \cref{thm:it_sk_forward} showing the forward stability of iterative sketching.
We conclude in \cref{sec:conclusions}.

\subsection{Notation} \label{sec:notation}

Throughout the paper, $\mat{A}\in\real^{m\times n}$, $m\ge n$ denotes a full-rank matrix and $\vec{b} \in \real^m$ denotes a vector.
The condition number of $\mat{A}$ is $\kappa = \kappa(\mat{A}) \coloneqq \sigma_{\rm max}(\mat{A}) / \sigma_{\rm min}(\mat{A})$.
For $\vec{y} \in \real^n$, we denote the residual $\vec{r}(\vec{y}) \coloneqq \vec{b} - \mat{A}\vec{y}$.
Double lines $\norm{\cdot}$ indicate either the Euclidean norm of a vector or the spectral norm of a matrix.
The symbol $u$ denotes the \emph{unit roundoff}, a measure of the size of rounding errors.
In double precision arithmetic, $u$ is roughly $10^{-16}$.

\subsection{Acknowledgements}
\label{sec:ack}

The suggestion to investigate the stability of iterative sketching was made to the author by Joel Tropp.
We thank Maike Meier, Riley Murray, Yuji Nakatsukasa, Christine Tobler, Joel Tropp, Marcus Webb, Robert Webber, Heiko Weichelt, and the anonymous referees for helpful discussions and feedback.

\section{Background}
\label{sec:background}

In this section, we quickly review some relevant background material: perturbation theory and numerical stability for least-squares problems (\cref{sec:pert}), random (subspace) embeddings (\cref{sec:subspace}), and the sketch-and-solve method (\cref{sec:sketch-and-solve}).

\subsection{Perturbation theory and numerical stability} \label{sec:pert}

The textbook method for least-squares problems \cref{eq:ls} is Householder \QR factorization.
Householder \QR factorization is backward stable \cite[Thm.~20.3]{Hig02}: In floating point arithmetic, it produces a numerical solution $\vec{\hat{x}}$ that is the exact solution to a slightly modified problem:
\begin{subequations} \label{eq:ls_backward} 
  \begin{equation} 
    \vec{\hat{x}} = \argmin_{\vec{y} \in \real^n} \norm{(\vec{b}+\dta{b}) - (\mat{A}+\dta{A})\vec{y}},
  \end{equation}
  where the (relative) size of the perturbations is at most
  \begin{equation} \label{eq:Ab_perturbations}
    \norm{\dta{A}} \le c u \norm{\mat{A}}, \quad \norm{\dta{b}} \le c u \norm{\vec{b}} \quad \text{provided } cu < 1.
  \end{equation}
\end{subequations}
Here, the prefactor $c$ is a low-degree polynomial in $m$ and $n$.

Our main result is that iterative sketching produces a solution $\vec{\hat{x}}$ which has a forward error $\norm{\vec{x} - \vec{\hat{x}}}$ and a residual error $\norm{\vec{r}(\vec{x})-\vec{r}(\vec{\hat{x}})}$ that is comparable to the errors produced by a backward stable method.
But what are the forward error and residual error for a backward stable method?
To answer this question, we can use Wedin's theorem \cite[Thm.~5.1]{Wed73}.
Here is a simplified version: 

\begin{fact}[Wedin's theorem] \label{thm:wedin}
  Let $\mat{A},\dta{A}\in\real^{m\times n}$ and $\vec{b},\dta{b}\in\real^m$ and set
  \begin{equation*}
    \vec{x} = \argmin_{\vec{y} \in \real^n} \norm{ \vec{b} - \mat{A}\vec{y}}, \quad \vec{\hat{x}} = \argmin_{\vec{y} \in \real^n} \norm{(\vec{b} + \dta{b}) - (\mat{A}+\dta{A})\vec{y}}.
  \end{equation*}
  Suppose that $\norm{\dta{A}} \le \varepsilon \norm{\mat{A}}$ and $\norm{\dta{b}} \le \varepsilon \norm{\vec{b}}$ for $\varepsilon \in (0,1)$. 
  Then if $\varepsilon \kappa(\mat{A}) \le 0.1$, 
  \begin{align*}
    \norm{\vec{x} - \vec{\hat{x}}} &\le 2.23\kappa(\mat{A}) \left( \norm{\vec{x}} + \frac{\kappa(\mat{A})}{\norm{\mat{A}}} \norm{\vec{r}(\vec{x})} \right) \varepsilon, \\
    \norm{\vec{r}(\vec{x}) - \vec{r}(\vec{\hat{x}})}&\le 2.23\left( \norm{\mat{A}}\norm{\vec{x}} + \kappa(\mat{A})\norm{\vec{r}(\vec{x})} \right)\varepsilon.
  \end{align*}
\end{fact}

By Wedin's theorem, assuming that $\kappa u$ is sufficiently small, the solution $\vec{\hat{x}}$ to a backward stable least-squares solver \cref{eq:ls_backward} satisfies
\begin{subequations} \label{eq:ls_forward}
  \begin{align}
    \norm{\vec{\hat{x}} - \vec{x}} &\le c' \kappa \left( \norm{\vec{x}} + \frac{\kappa}{\norm{\mat{A}}} \norm{\vec{r}(\vec{x})} \right) u, \label{eq:forward_stable} \\
    \norm{\vec{r}(\vec{\hat{x}}) - \vec{r}(\vec{x})} &\le c' \left( \norm{\mat{A}}\norm{\vec{x}} + \kappa\norm{\vec{r}(\vec{x})} \right) u, \label{eq:residual_error}
  \end{align}
\end{subequations}
where $c'$ denotes a small multiple of the constant $c$ appearing in \cref{eq:Ab_perturbations}.
A method is said to be \emph{forward stable} if the computed solution $\vec{\hat{x}}$ satisfies a bound of the form \cref{eq:ls_forward}.
For more on stability of numerical algorithms, the definitive reference is \cite{Hig02}.
The reference \cite{Bjo96} contains a comprehensive treatment of the least-squares problem and its numerical properties.

\subsection{Subspace embeddings}
\label{sec:subspace}

The core ingredient for randomized least-squares solvers are \emph{random embeddings} \cite[\S\S8--9]{MT20a}.

\begin{definition}[Subspace embedding]
  A matrix $\mat{S} \in \real^{d\times m}$ is said to be a \emph{subspace embedding} with \emph{distortion} $\varepsilon \in (0,1)$ for a subspace $\mathcal{V} \subseteq \real^m$ if
  \begin{equation} \label{eq:subspace_embedding}
    (1-\varepsilon)\norm{\vec{v}} \le \norm{\mat{S}\vec{v}} \le (1+\varepsilon) \norm{\vec{v}} \quad \text{for every }\vec{v} \in \mathcal{V}.
  \end{equation}
  A random matrix $\mat{S} \in \real^{d\times n}$ is an \emph{oblivious subspace embedding} for dimension $k$ if, for any $k$-dimensional subspace $\mathcal{V}$, the subspace embedding property \cref{eq:subspace_embedding} holds with probability at least $1-\delta$ for a specified \emph{failure probability} $\delta \in (0,1)$.
\end{definition}

There are a number of constructions for oblivious subspace embeddings.
For instance, a matrix populated with independent Gaussian entries with mean zero and variance $1/d$ is an oblivious subspace embedding with $d \approx k/\varepsilon^2$ (see \cite[Thm.~8.4]{MT20a} for a precise statement).
There are also fast embeddings that can be multiplied by a vector in $\order(m \log m)$ operations such subsampled randomized trigonometric transforms and sparse sign embeddings \cite[\S9]{MT20a}.

In this paper, we will use \emph{sparse sign embeddings} for numerical experiments, which were definitively the fastest in computational experiments reported in \cite[Fig.~2]{DM23}.
These embeddings take the form
\begin{equation} \label{eq:sparse_sign}
    \mat{S} \coloneqq \sqrt{\frac{1}{\zeta}} \begin{bmatrix} \vec{s}_1 & \cdots & \vec{s}_m \end{bmatrix},
\end{equation}
where each $\vec{s}_j \in \real^d$ is populated with exactly $\zeta$ nonzero entries with uniform $\pm 1$ values in uniformly random positions.
The theoretically sanctioned parameter choices \cite{Coh16} are $d = \order((k\log k)/\varepsilon^2)$ and $\zeta = \order((\log k)/\varepsilon)$, but more aggressive parameter choices see use in practice \cite[\S3.3]{TYUC19}.

The following singular value bounds (see, e.g., \cite[Prop.~5.4]{KT24}) will be useful later:

\begin{fact}[Singular value bounds] \label{fact:subspace_embedding}
  Let $\mat{S}$ be a subspace embedding for $\range(\mat{A})$ with distortion $\varepsilon \in (0,1)$ and let $\mat{Q}\mat{R} = \mat{S}\mat{A}$ be an economy \QR decomposition of $\mat{S}\mat{A}$.
  Then $\mat{R}$ satisfies the bounds
  \begin{equation*}
    \sigma_{\rm max}(\mat{R}) \le (1+\varepsilon) \sigma_{\rm max}(\mat{A}), \quad \sigma_{\rm min}(\mat{R}) \ge (1-\varepsilon)\sigma_{\rm min}(\mat{A}).
  \end{equation*}
  In addition, $\mat{A}\mat{R}^{-1}$ satisfies the bounds
  \begin{equation*}
    \sigma_{\rm max}(\mat{A}\mat{R}^{-1}) \le \frac{1}{1-\varepsilon}, \quad \sigma_{\rm min}(\mat{A}\mat{R}^{-1}) \ge \frac{1}{1+\varepsilon}.
  \end{equation*}
\end{fact}

\subsection{The sketch-and-solve method}
\label{sec:sketch-and-solve}

The sketch-and-solve method \cite{Sar06} can used to provide an approximate solution to the least-squares problem \cref{eq:ls}.
We use the sketch-and-solve method as the initial iterate $\vec{x}_0$ in our pseudocodes for both sketch-and-precondition (\cref{alg:spre}) and iterative sketching (\cref{alg:it-sk}) methods.
In the experiments of \cite{MNTW23} and our \cref{fig:other}, the sketch-and-solve initialization signficiantly improves the numerical stability of the sketch-and-precondition method.

Let $\mat{S}$ be a subspace embedding for $\range(\onebytwo{\mat{A}}{\vec{b}})$ with distortion $\varepsilon \in (0,1)$.
In sketch-and-solve, we compute an approximate solution $\vec{x}_0$ to the least-squares problem \cref{eq:ls} as follows:
\begin{equation} \label{eq:sketched_ls}
  \vec{x}_0 = \argmin_{\vec{y} \in \real^n} \norm{\mat{S}\vec{b} - (\mat{S}\mat{A})\vec{y}}.
\end{equation}
To compute $\vec{x}_0$ numerically, we use a (Householder) \QR factorization of the sketched matrix $\mat{S}\mat{A}$:
\begin{equation*}
  \mat{S}\mat{A} = \mat{Q}\mat{R}, \quad \vec{x}_0 = \mat{R}^{-1}(\mat{Q}^\top (\mat{S}\vec{b})).
\end{equation*}
We have the following (standard) guarantee for sketch-and-solve, proven in \cref{app:proofs}:

\begin{fact}[Sketch-and-solve] \label{fact:sketch-and-solve}
  With the present notation and assumptions, 
  \begin{gather*}
    \norm{\vec{r}(\vec{x}_0)} \le \frac{1+\varepsilon}{1-\varepsilon} \norm{\vec{r}(\vec{x})}, \quad \norm{\vec{r}(\vec{x}) - \vec{r}(\vec{x}_0)} \le \frac{2\sqrt{\varepsilon}}{1-\varepsilon} \norm{\vec{r}(\vec{x})}, \siamorarxiv{\\}{\quad} \norm{\vec{x} - \vec{x}_0} \le \frac{2\sqrt{\varepsilon}}{1-\varepsilon} \frac{\kappa}{\norm{\mat{A}}}\norm{\vec{r}(\vec{x})}.
  \end{gather*}
\end{fact}

\section{Iterative sketching} \label{sec:it-sk}

In this section, we introduce the iterative sketching method and develop a stable implementation.
Proof of its forward stability appears in \cref{sec:it-sk-for}.

\subsection{Description of method and formulation as iterative refinement}

In 2016, Pilanci and Wainwright introduced the iterative sketching method under the name \emph{iterative Hessian sketch} \cite{PW16a}.
Their paper frames the method as a sequence of minimization problems
\begin{equation*}
    \vec{x}_{i+1} = \argmin_{\vec{y} \in \real^n} \norm{(\mat{S}_i \mat{A})\vec{y}}^2 - \vec{y}^\top \mat{A}^\top (\vec{b} - \mat{A}\vec{x}_i) \quad \text{for } i=0,1,2,\ldots.
\end{equation*}
In this original paper, an independent embedding $\mat{S}_i$ is drawn at each iteration.
Subsequent papers \cite{OPA19,LP21} proposed a version of iterative sketching with a single embedding $\mat{S}$ (as in \cref{eq:it-sk-seq-min}) and developed variants of iterative sketching incorporating damping and momentum (see \cref{sec:damp-mom}).

For the purposes of this paper, it will be helpful to adopt an alternative interpretation of iterative sketching as \emph{iterative refinement on the normal equations}.
The normal equations 
\begin{equation} \label{eq:normal}
  (\mat{A}^\top\mat{A})\vec{x} = \mat{A}^\top\vec{b}.
\end{equation}
are a square system of equations satisfied by the solution $\vec{x}$ to the least-squares problem \cref{eq:ls}.
To develop a fast least-squares solver, we can approximate $\mat{A}^\top\mat{A}$ by using a subspace embedding:
\begin{equation*}
  (\mat{S}\mat{A})^\top (\mat{S}\mat{A}) \approx \mat{A}^\top\mat{A}.
\end{equation*}
The matrix $(\mat{S}\mat{A})^\top (\mat{S}\mat{A})$ can be used as a \emph{preconditioner} for an iterative method to solve the normal equations \cref{eq:normal}.
Specifically, iterative sketching uses one of the simplest iterative methods, (fixed-precision) iterative refinement \cite[Ch.~12]{Hig02}.
This leads to an iteration
\begin{equation} \label{eq:it-sk-it-ref}
  (\mat{S}\mat{A})^\top (\mat{S}\mat{A}) \vec{d}_i = \mat{A}^\top (\vec{b} - \mat{A}\vec{x}_i), \quad \vec{x}_{i+1} \coloneqq \vec{x}_i + \vec{d}_i.
\end{equation}
The iteration \cref{eq:it-sk-it-ref} is another description of the iterative sketching method, equivalent to the formulation based on sequential minimization \cref{eq:it-sk-seq-min} from the introduction.
We choose our initial iterate $\vec{x}_0$ to be the sketch-and-solve solution \cref{eq:sketched_ls}.

In exact arithmetic, iterative sketching converges geometrically.
\begin{theorem}[Convergence of iterative sketching] \label{thm:it-sk-convergence}
  Let $\mat{S}$ be a subspace embedding with distortion $0 < \varepsilon < 1-1/\sqrt{2}$. 
  Then the iterates $\vec{x}_0,\vec{x}_1,\ldots$ produced by the iterative sketching method (\cref{alg:it-sk}) in exact arithmetic satisfy the following bounds:
  \begin{equation*}
    \norm{\vec{x} - \vec{x}_i} < (8-2\sqrt{2})\sqrt{\varepsilon}\kappa \,g_{\rm IS}^i\, \frac{\norm{\vec{r}(\vec{x})}}{\norm{\mat{A}}}, \quad \norm{\vec{r}(\vec{x}) - \vec{r}(\vec{x}_i)} < (8-2\sqrt{2})\sqrt{\varepsilon}\,g_{\rm IS}^i \,\norm{\vec{r}(\vec{x})}.
  \end{equation*}
  The convergence rate $g_{\rm IS}$ is
  \begin{equation} \label{eq:itsk-rate}
    g_{\rm IS} = \frac{(2-\varepsilon)\varepsilon}{(1-\varepsilon)^2} \le (2+\sqrt{2})\varepsilon.
  \end{equation}
\end{theorem}
We present the proof at the end of this section, as it will parallel out forthcoming analysis of iterative sketching in floating point arithmetic.
This result demonstrates that a distortion of $\varepsilon \approx 0.29$ is necessary to achieve convergence.
Assuming $\varepsilon \approx \sqrt{n/d}$, this implies \textbf{an embedding dimension of $d \approx 12 n$ is needed to obtain a convergent scheme}.
The large required embedding dimension $d\gtrapprox 12n$ is a genuine property of the iterative sketching method, not an artifact of the analysis.

\subsection{Numerically stable implementation}

In this section, we develop a stable implementation of iterative sketching, which was presented in the introduction as \cref{alg:it-sk}.
This fills a gap in the literature, as previous pseudocodes for iterative sketching in the literature \cite{PW16a,OPA19,LP21} do not offer precise guidance for how to stably implement the recurrence \cref{eq:it-sk-it-ref}.

To stably solve the recurrence \cref{eq:it-sk-it-ref}, we compute an economy \QR factorization
\begin{equation} \label{eq:it-sk-qr}
  \mat{S}\mat{A} = \mat{Q}\mat{R}.
\end{equation}
As a consequence, we have \emph{implicitly} computed a Cholesky decomposition:
\begin{equation*}
  (\mat{S}\mat{A})^\top (\mat{S}\mat{A}) = \mat{R}^\top \mat{R}, 
\end{equation*}
allowing us to solve \cref{eq:it-sk-it-ref} using two triangular solves, i.e., 
\begin{equation*}
  \mat{R}^\top \mat{R} \vec{d}_i = \mat{A}^\top (\vec{b} - \mat{A}\vec{x}_i) \implies \vec{d}_i = \mat{R}^{-1} \left( \mat{R}^{-\top} \left(\mat{A}^\top \left(\vec{b} - \mat{A}\vec{x}_i\right)\right)\right).	
\end{equation*}
A stable implementation of iterative sketching using this approach is presented as \cref{alg:it-sk}.

To illustrate the importance of our stable implementation (\cref{alg:it-sk}), we show the numerical performance of three ``bad'' implementations in \cref{fig:bad}:
\begin{itemize}
\item \textbf{Bad matrix.} In this implementation, we form the matrix $(\mat{S}\mat{A})^\top (\mat{S}\mat{A})$ explicitly and factorize it using Cholesky decomposition or, should Cholesky fail (which it typically does, if $\kappa \gtrapprox 1/\sqrt{u}$), \LU factorization with partial pivoting.
  We then use this factorization to solve \cref{eq:it-sk-it-ref}.
(\emph{Stable version:} use \QR factorization \cref{eq:it-sk-qr}.)
\item \textbf{Bad residual.} In this implementation, we evaluate the right-hand side of \cref{eq:it-sk-it-ref} using the algebraically equivalent---but not numerically equivalent---expression $\mat{A}^\top \vec{b} - \mat{A}^\top (\mat{A}\vec{x})$.
(\emph{Stable version:} use $\mat{A}^\top(\vec{b}-\mat{A}\vec{x})$.)
\item \textbf{Bad initialization.} In this implementation, we set $\vec{x}_0 \coloneqq \vec{0}$. 
(\emph{Better version:} use the sketch-and-solve initialization \cref{eq:sketched_ls}.)
\end{itemize}
We evaluate these bad implementations using the same setup as \cref{fig:numerical} with $\kappa = 10^{10}$ and $\norm{\vec{r}(\vec{x})} = 10^{-6}$.
The first \textbf{bad matrix} fails to converge at all, with the numerical solution diverging at an exponential rate.
The \textbf{bad residual} is not stable, having significantly higher error in all metrics than the stable implementation in \cref{alg:it-sk} and the reference accuracy of Householder \QR.
The \textbf{bad initialization} method is stable, but takes over twice as long to converge as \cref{alg:it-sk}.

\begin{figure}[t]
  \centering
  
  \begin{subfigure}{0.45\textwidth}
    \includegraphics[width=\textwidth]{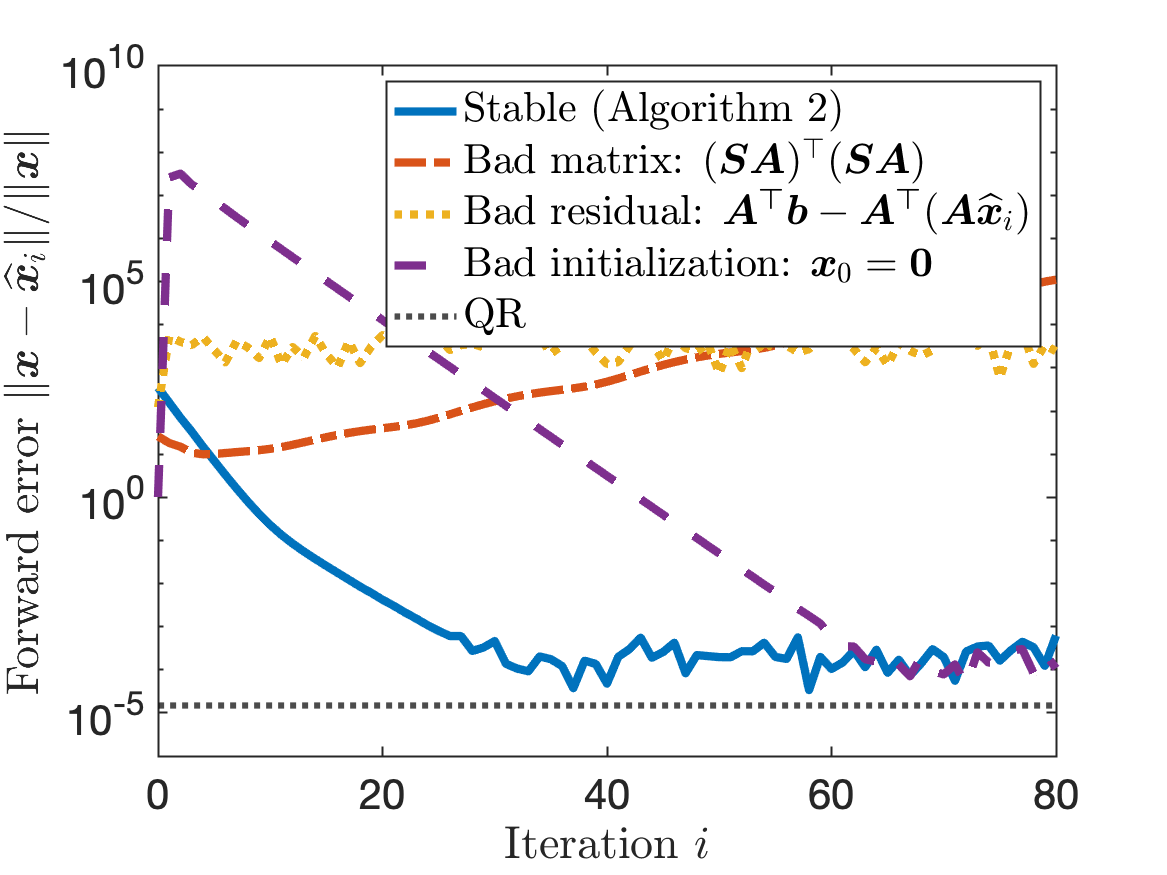}
  \end{subfigure}
  \begin{subfigure}{0.45\linewidth}
    \includegraphics[width=\textwidth]{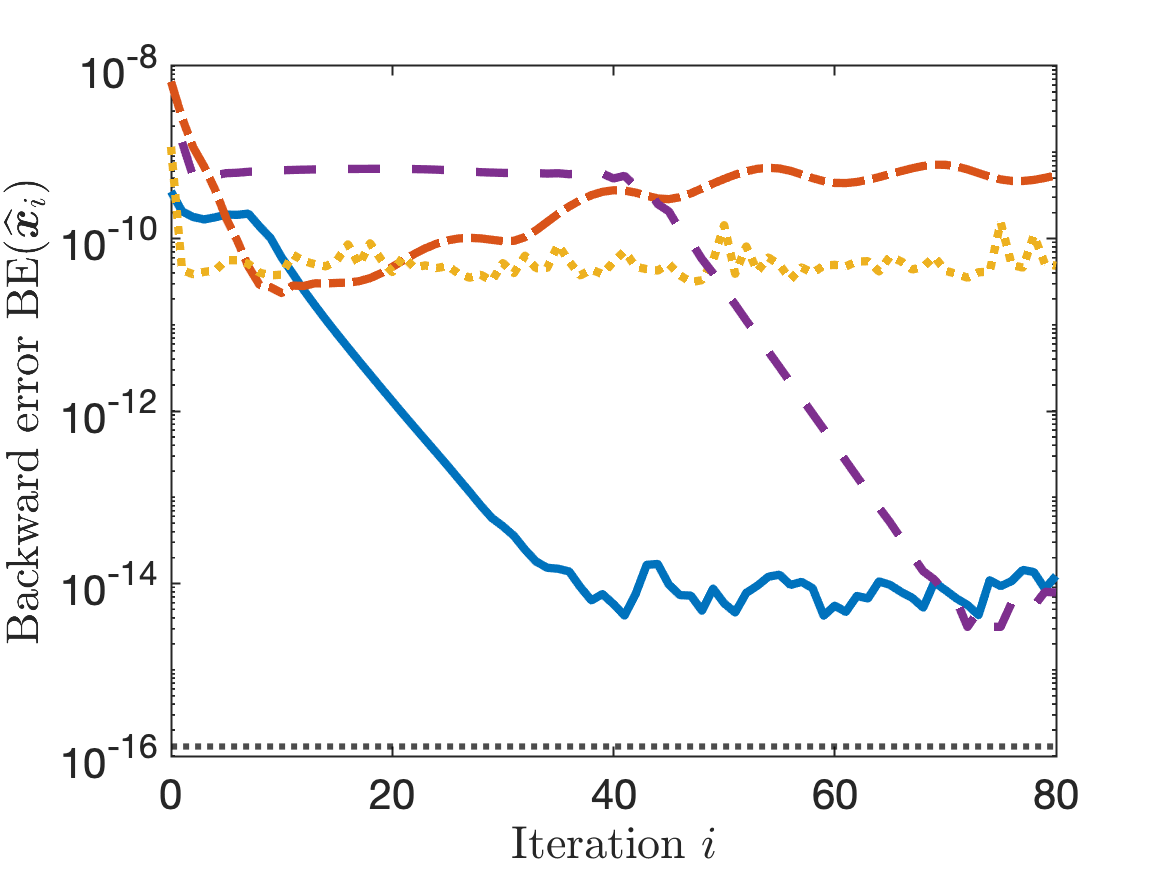}
  \end{subfigure}
  
  \caption{Comparison of the stable implementation of iterative sketching (\cref{alg:it-sk}, blue solid) to three ``bad'' implementations, referred to as \textbf{bad matrix} (red dash-dotted), \textbf{bad residual} (yellow dotted), and \textbf{bad initial} (purple dashed).
    The accuracy of Householder \QR factorization is provided as a baseline (black dotted).} \label{fig:bad}
  
\end{figure}

\subsection{Implementation guidance} \label{sec:implementation}

In this section, we provide our practical implementation guidance for iterative sketching.
As always, implementation details should be adapted to the format of the input matrix $\mat{A}$ (dense, sparse,\ldots) and the available computing resources.
Our recommendations for iterative sketching are targeted at dense $\mat{A}$ for computation on a laptop-scale computer, and are provided as a baseline from which to do machine-specific tuning.
See \cite{CDD+23} for a recent approach to autotuning randomized least-squares problems.

Based on the runtime evaluation in \cite{DM23}, we recommend choosing $\mat{S}$ to be a sparse sign embedding. 
Following \cite[\S3.3]{TYUC19}, we use sparsity level $\zeta = 8$.

If one is using iterative sketching to solve to machine accuracy, we recommend setting
\begin{equation} \label{eq:dimension}
  d \coloneqq \max\left(\left\lceil (6+4\sqrt{2})n \cdot \exp\left(\mathrm{W}\left((6-4\sqrt{2})\frac{m}{n^2} \log \left(\frac{1}{u}\right) \right)\right)\right\rceil, 20n\right).
\end{equation}
Here, $\rm W$ is the Lambert W function and $u$ is the unit roundoff.
To solve to less than machine accuracy, $u$ can be replaced by the desired accuracy level in \cref{eq:dimension}.
We obtained this expression by balancing the $2dn^2$ operations required to \QR factorize $\mat{S}\mat{A}$ and the roughly $4mn\log(1/u)/\log(1/g_{\rm IS})$ operations needed to iterate until convergence (in the worst case).
We choose to enforce $d \ge 20n$ to ensure rapid convergence even when $m$ and $n$ are close to each other.

To achieve nearly the maximum numerically achievable accuracy, we can use Wedin's theorem (\cref{thm:wedin}) to inform a stopping criterion for the algorithm.
As \cref{fact:subspace_embedding} shows, $\norm{\mat{R}}$ and $\kappa(\mat{R}) \coloneqq \sigma_{\rm max}(\mat{R})/\sigma_{\rm min}(\mat{R})$ are within a small multiple of $\norm{\mat{A}}$ and $\kappa$, motivating the following procedure:
\begin{itemize}
\item Produce an estimate $\mathrm{normest}$ for $\norm{\mat{R}} \approx \norm{\mat{A}}$ by applying $\lceil \log n \rceil$ steps of the randomized power method \cite[\S6.2]{MT20a} to $\mat{R}$.
\item Produce an estimate $\mathrm{condest}$ for $\kappa(\mat{R}) \approx \kappa$ using condition number estimation \cite[Ch.~15]{Hig02}.
\item Run the iteration until the change in residual is small compared to the final residual accuracy predicted by Wedin's theorem\siamorarxiv{. Specifically, we recommend the stopping criteria}{:}
  \begin{equation*}
    \iarxiv{\text{STOP} \quad \text{when} \quad} \norm{\vec{r}(\vec{x}_{i+1}) - \vec{r}(\vec{x}_i)} \le u \left( \gamma \cdot \mathrm{normest} \cdot \norm{\vec{x}_{i+1}} + \rho \cdot \mathrm{condest} \cdot \norm{\vec{r}(\vec{x}_{i+1})} \right).
  \end{equation*}
  We use $\gamma = 1$ and $\rho = 0.04$.
\end{itemize}
We find this procedure reliably leads to solutions whose forward and residual errors are a small multiple of the solution produced by Householder \QR.
If more accuracy is desired, one can run the iteration for a handful of additional iterations after the tolerance is met.
One can use a looser tolerance if far less accuracy than Householder \QR is needed.

\subsection{Numerical example: Dense kernel regression problems}
\label{sec:application}

We now compare iterative sketching to MATLAB's \QR-based solver \texttt{mldivide} on a dense kernel regression least-squares problem.
Our goal is to confirm the $\order((mn+n^3) \log n)$ complexity and forward stability of iterative sketching; larger-scale experiments could be a subject for future work.
Suppose we want to learn an unknown function $f : \real^k \to \real$ from input--output pairs $(\vec{z}_1,b_1),\ldots,(\vec{z}_m,b_m) \in \real^k \times \real$.
We can solve this problem using \emph{restricted kernel ridge regression} (KRR) \cite{SB00,RCR17,DEF+23}.
Let $K : \real^k \times \real^k \to \real$ be a positive definite kernel function, $\mu > 0$ a regularization parameter, and $\set{S} = \{s_1,\ldots,s_n\} \subseteq \{1,\ldots,m\}$ a subset of $n$ (uniformly random) data points.
In restricted KRR, we assume an ansatz
\begin{equation*}
  f(\cdot) = \sum_{j=1}^n K(\cdot, \vec{z}_{s_j}) x_j.
\end{equation*}
The coefficients $\vec{x} = (x_1,\ldots,x_n) \in \real^n$ are chosen to minimize the regularized least-squares loss
\begin{equation*}
  \operatorname{Loss}(\vec{x}) \coloneqq \norm{\vec{b} - \mat{A}\vec{x}}^2 + \mu \, \vec{x}^\top  \mat{H}\vec{x} \quad \text{where } a_{ij} = K(\vec{z}_i, \vec{z}_{s_j}) \text{ and } h_{ij} = K(\vec{z}_{s_i},\vec{z}_{s_j}).
\end{equation*}
For simplicity, we set the regularization to zero, $\mu \coloneqq 0$, which reduces the restricted KRR problem to an ordinary linear least-squares problem.

To perform our evaluation, we use a sample of $m = 10^6$ data points from the SUSY dataset \cite{BSW14}.
We standardize the data and use a square-exponential kernel
\begin{equation*}
  K(\vec{z},\vec{z}') = \exp \left( - \frac{\norm{\vec{z} - \vec{z}'}^2}{2\sigma^2} \right) \quad \text{with } \sigma = 4.
\end{equation*}
We randomly subsample $1000$ points $s_1,\ldots,s_{1000} \in \{1,\ldots,10^6\}$ and test using values of $n = |\set{S}|$ between $10^1$ and $10^3$.
We implement iterative sketching following the guidance in \cref{sec:implementation}.

\begin{figure}[t]
  \centering
  
  \includegraphics[width=0.48\textwidth]{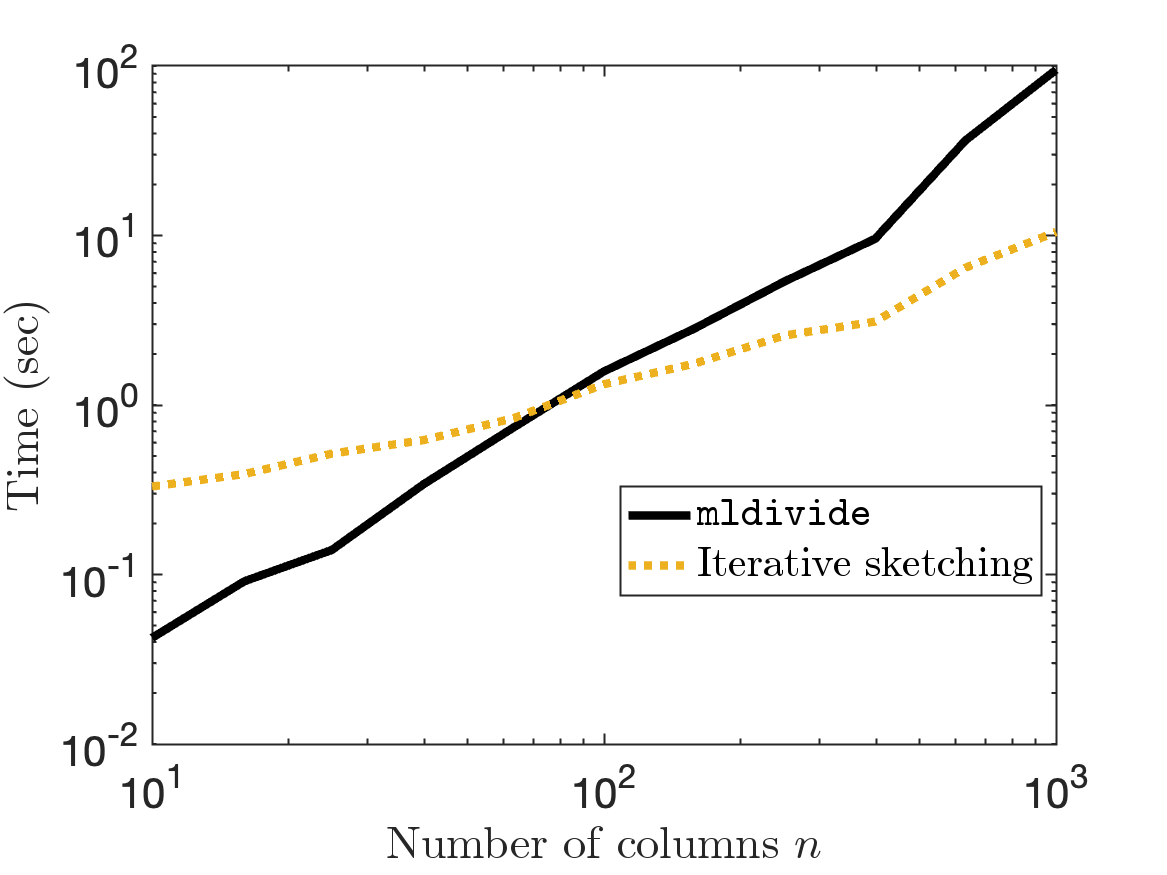}
  \hfill
  \includegraphics[width=0.48\textwidth]{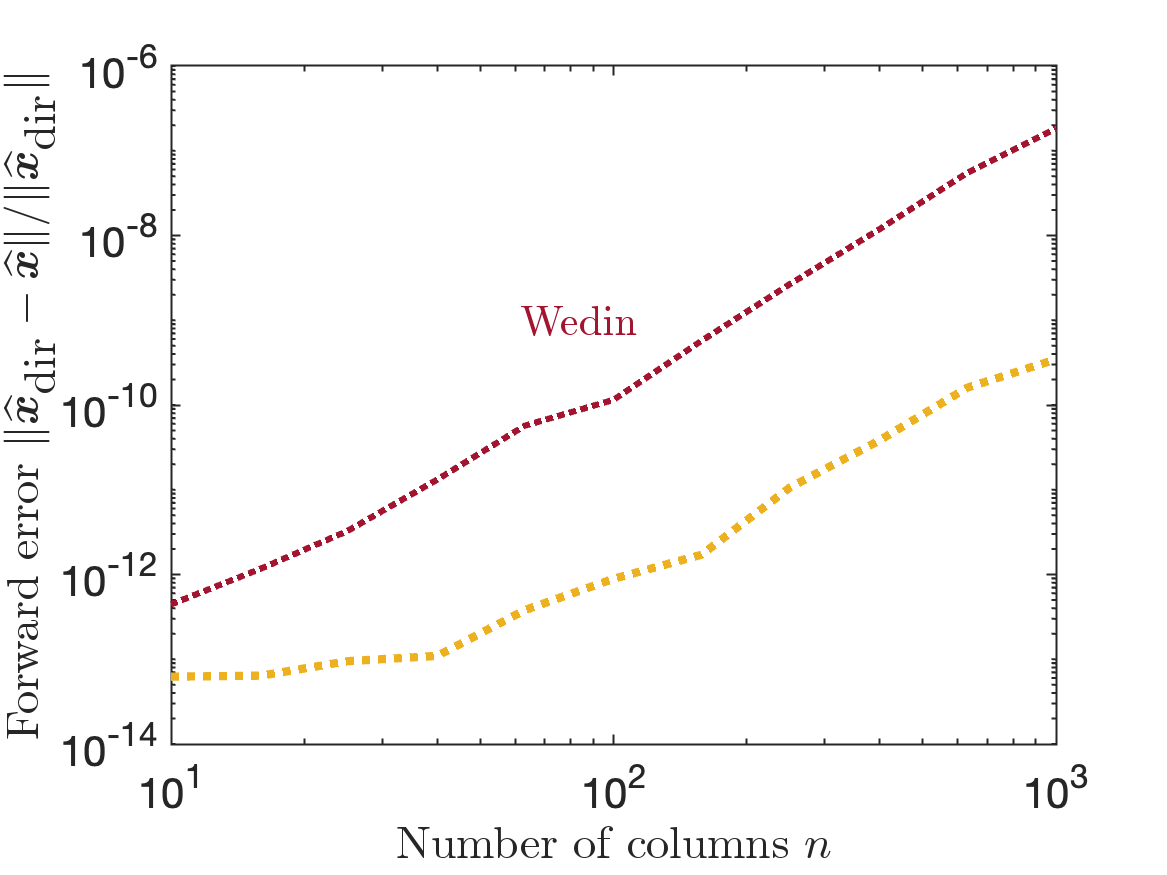} 
    
  \caption{\emph{Left:} Runtime of iterative sketching and MATLAB's \texttt{mldivide} for the kernel regression problem for numbers of columns $n \in [10^1,10^3]$. \emph{Right:} Relative difference of solution $\vec{\hat{x}}$ by iterative sketching and $\vec{\hat{x}}_{\rm dir}$ by MATLAB's \texttt{mldivide}. The upper bound provided by Wedin's theorem (\cref{thm:wedin}) with $\varepsilon = u$ is shown as a dotted maroon line.} \label{fig:susy}
\end{figure}

\Cref{fig:susy} shows the results.
The left panel shows the runtime.
As $n$ increases, iterative sketching becomes faster than MATLAB's \texttt{mldivide}, achieving \textbf{a 10$\times$ speedup} at $n = 1000$.
The right panel shows the norm of the difference between the solutions computed by \texttt{mldivide} and iterative sketching; the upper bound on the forward error by Wedin's theorem (\cref{thm:wedin}) is shown for reference.
With our chosen parameter settings and stopping criteria, iterative sketching is numerically stable on this problem, achieving a level of accuracy consistent with Wedin's theorem.

\subsection{Numerical example: Sparse problems} \label{sec:sparse}

Iterative sketching can also be applied to sparse least-square problems, with large speedups over direct solversendend for problems lacking an efficient elimination ordering.
We illustrate with a synthetic example.
Set $m\coloneqq 3\times 10^6$ and consider a range of values $n$ between $10^1$ and $10^3$. 
For each value of $n$, we generate a matrix $\mat{A}\in\real^{m\times n}$ with three nonzeros per row placed in uniformly random positions with uniformly random values $\pm 1$.
We generate $\vec{b} \in \real^m$ with independent standard Gaussian entries, and set the embedding dimension to $d=30n$.
The left panel of \cref{fig:sparse} shows the runtime of iterative sketching and \texttt{mldivide}.
Iterative sketching is faster than \texttt{mldivide} for $n\ge 50$, obtaining a \textbf{a 36$\times$ speedup} at $n = 1000$.

Another advantage of iterative sketching over \texttt{mldivide} for sparse problems is memory usage.
To observe the memory usage, we use the unix command \texttt{ps -p \$PID -o rss=}.
Results are shown in the right panel of \cref{fig:sparse}.
For the direct method, the memory usage climbs steadily with $n$, whereas for iterative sketching memory usage remains steady, in fact dropping slightly (possibly an artifact of our profiling strategy).
The memory usage of these two methods has a big impact on the size of problems that can be solved.
On our machine, MATLAB crashes when we try and solve a $10^8\times 10^4$ sparse least-squares problem, where iterative sketching succeeds.
The ability to solve very large sparse least-squares problems to high accuracy may be one of the most compelling use cases for iterative sketching and other randomized iterative methods.

\begin{figure}[t]
  \centering
  
  \includegraphics[width=0.48\textwidth]{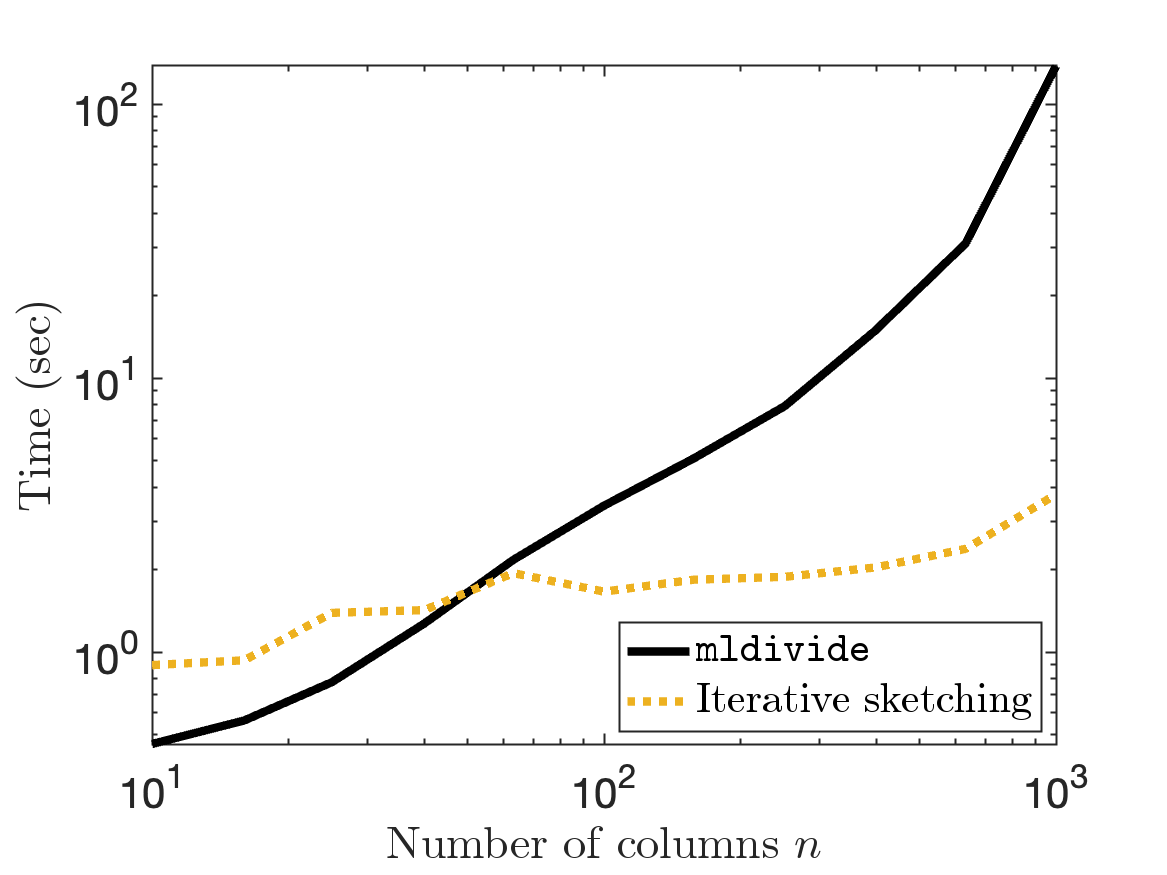}
  \includegraphics[width=0.48\textwidth]{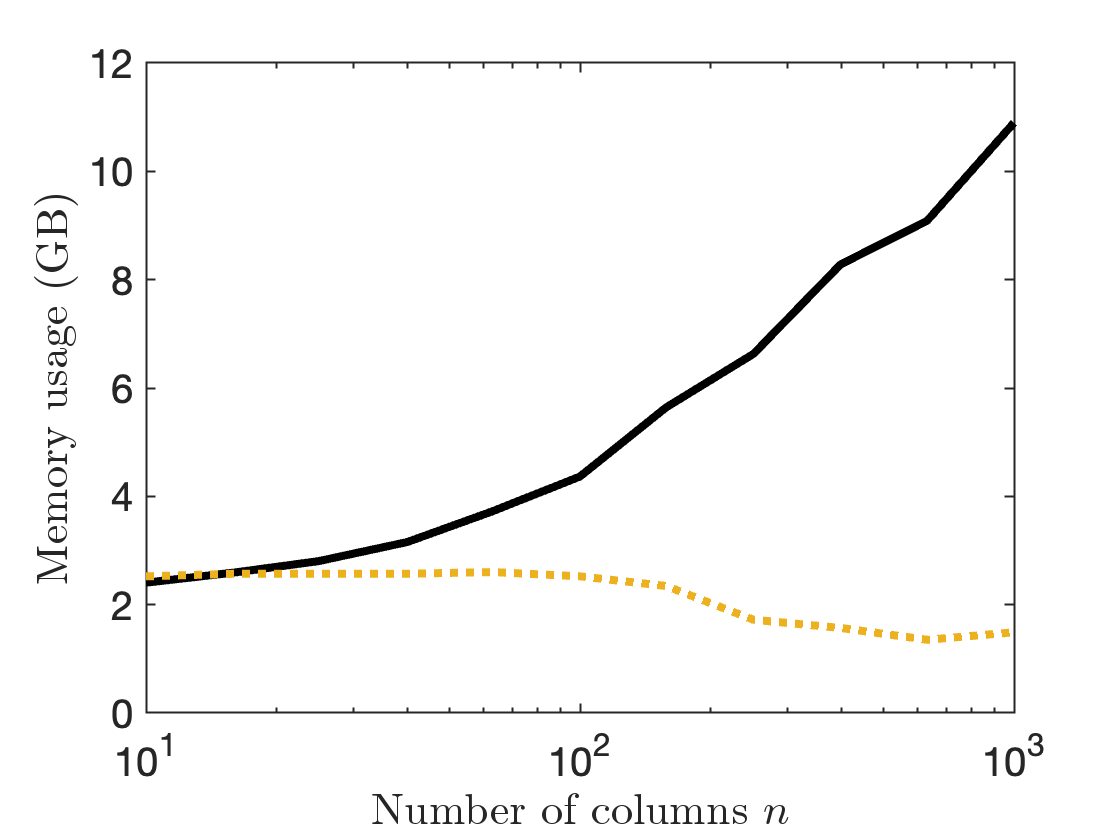}
    
  \caption{Runtime (\emph{left}) and memory usage (\emph{right}) of iterative sketching and MATLAB's \texttt{mldivide} for the sparse problem for numbers of columns $n \in [10^1,10^3]$.} \label{fig:sparse}
\end{figure}

\subsection{\texorpdfstring{Proof of \cref{thm:it-sk-convergence}}{Proof of Theorem 6}} \label{sec:it-sk-analysis}

Let $\mat{R}$ be the R-factor of $\mat{S}\mat{A}$, as in \cref{eq:it-sk-qr}.
Then, with some rearranging, the iterative sketching recurrence \cref{eq:it-sk-it-ref} takes the form
\begin{equation*}
  \vec{x}_{i+1} = (\Id - \mat{R}^{-1}\mat{R}^{-\top} \mat{A}^\top\mat{A}) \vec{x}_i + \mat{R}^{-1}\mat{R}^{-\top}\mat{A}^\top\vec{b}.
\end{equation*}
The solution $\vec{x}$ is a fixed point of this iteration:
\begin{equation*}
  \vec{x} = (\Id - \mat{R}^{-1}\mat{R}^{-\top} \mat{A}^\top\mat{A}) \vec{x} + \mat{R}^{-1}\mat{R}^{-\top}\mat{A}^\top\vec{b}.
\end{equation*}
Subtracting the two previous displays leads to a recurrence for the error $\vec{x} - \vec{x}_i$:
\begin{equation*}
  \vec{x} - \vec{x}_{i+1} = \mat{R}^{-1}\mat{G} \mat{R}(\vec{x} - \vec{x}_i) \quad \text{where} \quad \mat{G} \coloneqq \Id - \mat{R}^{-\top} \mat{A}^\top\mat{A}\mat{R}^{-1}
\end{equation*}
and, consequently,
\begin{equation*}
  \vec{x} - \vec{x}_i = \mat{R}^{-1} \mat{G}^i \mat{R}(\vec{x} - \vec{x}_0).
\end{equation*}

Multiplying by $\mat{Q}\mat{R} = \mat{S}\mat{A}$ and taking norms, we obtain
\begin{equation*} \label{eq:it-sk-analysis-1}
  \norm{\mat{S}\mat{A}(\vec{x} - \vec{x}_i)} \le \norm{\mat{Q}}\norm{\mat{G}}^i \norm{\mat{R}(\vec{x} - \vec{x}_0)} = \norm{\mat{G}}^i \norm{\mat{S}\mat{A}(\vec{x} - \vec{x}_0)}.
\end{equation*}
Using the singular value bounds in \cref{fact:subspace_embedding}, we see that
\begin{equation*}
  \norm{\mat{G}} \le \max \left\{ \left(\frac{1}{1-\varepsilon}\right)^2-1,1-\left(\frac{1}{1+\varepsilon}\right)^2 \right\} = \frac{(2-\varepsilon)\varepsilon}{(1-\varepsilon)^2} = g_{\rm IS}.
\end{equation*}
Recall that we have chosen our initial iterate $\vec{x}_0$ to be the sketch-and-solve solution.
Therefore, we can invoke the subspace embedding property and the sketch-and-solve bound, \cref{fact:sketch-and-solve}, to obtain
\begin{equation*} \label{eq:sketch-and-solve-initial-bound}
  \norm{\mat{S}\mat{A}(\vec{x} - \vec{x}_0)} \le (1+\varepsilon) \norm{\mat{A}(\vec{x} - \vec{x}_0)} = (1+\varepsilon) \norm{\vec{r}(\vec{x}) - \vec{r}(\vec{x}_0)} \le 2\cdot\frac{1+\varepsilon}{1-\varepsilon}\cdot\sqrt{\varepsilon} \norm{\vec{r}(\vec{x})}.
\end{equation*}
Combining the three previous displays with one more application of the subspace embedding property, we obtain
\begin{equation} \label{eq:it-sk-res-bound}
    \begin{split}
  \norm{\vec{r}(\vec{x}) - \vec{r}(\vec{x}_i)} &\le \frac{1}{1-\varepsilon} \norm{\mat{S}\mat{A}(\vec{x} - \vec{x}_i)} \le \frac{1}{1-\varepsilon} \norm{\mat{G}}^i\norm{\mat{S}\mat{A}(\vec{x} - \vec{x}_0)}\\
  &\le 2\cdot\frac{1+\varepsilon}{(1-\varepsilon)^2}\cdot\sqrt{\varepsilon} g_{\rm IS}^i\,\norm{\vec{r}(\vec{x})}< (8-2\sqrt{2})\sqrt{\varepsilon} \, g_{\rm IS}^i\, \norm{\vec{r}(\vec{x})}.
  \end{split}
\end{equation}
In the final line, we use the hypothesis $\varepsilon< 1-1/\sqrt{2}$.
To prove the bound on $\norm{\vec{x} - \vec{x}_i}$, realize that
\begin{equation*}
  \norm{\vec{r}(\vec{x}) - \vec{r}(\vec{x}_i)} = \norm{\mat{A}(\vec{x} - \vec{x}_i)} \ge \sigma_{\rm min}(\mat{A}) \norm{\vec{x} - \vec{x}_i}.
\end{equation*}
Rearranging and applying \cref{eq:it-sk-res-bound} yields the stated bound on $\norm{\vec{x} - \vec{x}_i}$. \hfill $\proofbox$

\section{Iterative sketching is forward stable}
\label{sec:it-sk-for}

In this section, we prove that iterative sketching is forward stable:

\begin{theorem}[Iterative sketching is forward stable] \label{thm:it_sk_forward}
  Let $\mat{S}\in\real^{d\times m}$ be a subspace embedding for $\range(\onebytwo{\mat{A}}{\vec{b}})$ with distortion $\varepsilon\in(0,0.29]$.
  There exists a constant $c_1>0$ depending polynomially on $m$, $n$, and $d$ such that if $c_1 \kappa u < 1$ and multiplication by $\mat{S}$ is forward stable (in the precise sense outlined below \cref{eq:S_forward_stable}), then the numerically computed iterates $\vec{\hat{x}}_i$ by \cref{alg:it-sk} satisfy bounds
  \begin{equation} \label{eq:guarantees}
    \begin{split}
      \norm{\vec{x} - \vec{\hat{x}}_i} &\le 20\sqrt{\varepsilon} \kappa \left( g_{\rm IS} + c_1 \kappa u \right)^i \frac{\norm{\vec{r}(\vec{x})}}{\norm{\mat{A}}} + c_1 \kappa u \left[ \norm{\vec{x}} + \frac{\kappa}{\norm{\mat{A}}} \norm{\vec{r}(\vec{x})} \right], \\
      \norm{\vec{r}(\vec{x}) - \vec{r}(\vec{\hat{x}}_i)} &\le 20\sqrt{\varepsilon} \left( g_{\rm IS} + c_1 \kappa u \right)^i \norm{\vec{r}(\vec{x})} + c_1 u \left[ \norm{\mat{A}} \norm{\vec{x}} + \kappa \norm{\vec{r}(\vec{x})} \right].
    \end{split}
  \end{equation}
  Here, $g_{\rm IS}$ is the convergence rate of iterative sketching \cref{eq:itsk-rate}.
  In particular, if $g_{\rm IS} + c_1\kappa u \le 0.9$ and $\mat{S}$ is a sparse sign embedding \cref{eq:sparse_sign} with $d = \order(n \log n)$ and $\zeta = \order(\log n)$, then iterative sketching produces a solution $\vec{\hat{x}}$ that satisfies the forward stable error guarantee \cref{eq:ls_forward} in $q = \order(\log(1/u))$ iterations and $\order(mn \log (n/u)+ n^3 \log n)$ operations.
\end{theorem}

We stress that the bounds \cref{eq:guarantees} hold for any subspace embedding $\mat{S}$ satisfying the hypotheses, not just the specific embedding suggested in \cref{sec:implementation}.

The analysis is similar to the analysis of one step of iterative refinement on the seminormal equations by Bj\"ork \cite{Bjo87}.
The main innovations are the extension of the analysis to multiple steps of refinement and the incorporation of randomized dimensionality reduction.
We expect that this analysis will extend fairly directly to iterative sketching with damping or momentum (\cref{sec:damp-mom}).

This section constitutes a proof of \cref{thm:it_sk_forward}.
Throughout this section, we use the following notation:
Hats denote the numerically computed quantities in \cref{alg:it-sk}.
We use $\C,\C,\ldots$ to denote constants, depending polynomially on $m$, $n$, and the sketching dimension $d$ but no other problem parameters.
\resetconstant
We make the standing assumption that $\Cr{const:c1} \kappa u < 1$ where the constant $\Cl{const:c1}$ will at various points will be taken to be as large as needed to make the analysis work.
To simplify the presentation, we make no effort to track or optimize the implicit constants in this analysis.

\subsection{The computed R-factor}

We have assumed multiplication by $\mat{S}$ is forward stable, which means
\begin{equation} \label{eq:S_forward_stable}
  \mat{\hat{B}} = \mat{B} + \dta{B}, \quad \vec{\hat{Sb}} = \vec{Sb} + \vec{\delta}, \quad \norm{\dta{B}} \le \Cr{const:c1} \norm{\mat{A}} u, \quad \norm{\vec{\delta}} \le \Cr{const:c1} \norm{\vec{b}} u.
\end{equation}
With this stability assumption in place, we study the numerical accuracy of the computed R-factor $\mat{\hat{R}}$.
By the backwards stability of Householder QR factorization \cite[Thm.~19.4]{Hig02}, there exists a matrix $\mat{U}\in\real^{d\times n}$ with orthonormal columns such that
\begin{equation} \label{eq:B_pert}
  \mat{\hat{B}} = \mat{U}\mat{\hat{R}} \eqqcolon \mat{U} (\mat{R}+\dta{R}), \quad \norm{\dta{R}} \le \C \norm{\mat{A}} u.
\end{equation}
Consequently, we have
\begin{equation} \label{eq:R_inv_pert}
  \mat{\hat{R}}^{-1} = \mat{R}^{-1} (\Id + \mat{\Gamma}), \quad \norm{\mat{\Gamma}} \le \Cl{const:R_inv} \kappa u. 
\end{equation}
This leads to singular value bounds analogous to \cref{fact:subspace_embedding} in floating point:
\begin{proposition}[Singular value bounds, floating point] \label{prop:singular_values_fp}
  With the present setting and assumptions,
  \begin{equation*}
    \sigma_{\rm max}(\mat{\hat{R}}) \le (1 + \varepsilon + \Cl{const:svfp} u) \sigma_{\rm max}(\mat{A}), \quad \sigma_{\rm min}(\mat{\hat{R}}) \ge (1 - \varepsilon - \Cr{const:svfp} u) \sigma_{\rm min}(\mat{A}).
  \end{equation*}
  In addition, we have
  \begin{equation*}
    \sigma_{\rm max}(\mat{A}\mat{\hat{R}}^{-1}) \le \frac{1}{1-\varepsilon} + \Cr{const:svfp} \kappa u, \quad \sigma_{\rm min}(\mat{A}\mat{\hat{R}}^{-1}) \ge \frac{1}{1+\varepsilon} - \Cr{const:svfp} \kappa u.
  \end{equation*}
\end{proposition}

\begin{proof}
  This result is immediate from the singular value bounds in exact arithmetic (\cref{fact:subspace_embedding}), the bounds \cref{eq:B_pert,eq:R_inv_pert}, and the perturbation theorems of Weyl \cite[Thm.~III.2.1]{Bha97} and Gelfand--Naimark \cite[Thm.~III.4.5]{Bha97}.
\end{proof}

\subsection{Accuracy of initial iterate}
\label{sec:sketch-solve-numerical}

Our initial iterate $\vec{x}_0$ for iterative sketching is constructed by the sketch-and-solve method.
Therefore, we begin by an assessment of the numerical accuracy of the computed solution $\vec{\hat{x}}_0$.
For reasons that will become clear later, we will need a bound on $\norm{\mat{\hat{R}}(\vec{\hat{x}}_0 - \vec{x})}$.

\begin{lemma}[Bound on initial error] \label{lem:initial-error}
  With the present notation and assumptions,
  \begin{equation} \label{eq:initial-bound}
    \big\|\mat{\hat{R}}(\vec{x} - \vec{\hat{x}}_0 )\big\| \le 2.82\sqrt{\varepsilon} \norm{\vec{r}(\vec{x})} + \C u ( \norm{\mat{A}}\norm{\vec{x}} + \kappa \norm{\vec{r}(\vec{x})}).
  \end{equation}
\end{lemma}

\begin{proof}
  First, use the triangle inequality to bound
\siamorarxiv{
  \begin{equation} \label{eq:initial-error-bound}
    \begin{split}
      \big\|\mat{\hat{R}}(\vec{x}-\vec{\hat{x}}_0)\big\| &\le \norm{\dta{R}} \norm{\vec{\hat{x}}_0 - \vec{x}_0} + \norm{\mat{R}(\vec{\hat{x}}_0 - \vec{x}_0)} \\&\qquad+ \norm{\mat{R}(\vec{x}_0 - \vec{x})} + \norm{\dta{R}} \norm{\vec{x}_0 - \vec{x}}.
      \end{split}
    \end{equation}
}{
  \begin{equation} \label{eq:initial-error-bound}
    \begin{split}
      \big\|\mat{\hat{R}}(\vec{x}-\vec{\hat{x}}_0)\big\| &\le \norm{\dta{R}} \norm{\vec{\hat{x}}_0 - \vec{x}_0} + \norm{\mat{R}(\vec{\hat{x}}_0 - \vec{x}_0)} + \norm{\mat{R}(\vec{x}_0 - \vec{x})} + \norm{\dta{R}} \norm{\vec{x}_0 - \vec{x}}.
      \end{split}
  \end{equation}
}
  We now bound each of the terms on the right-hand side of \cref{eq:initial-error-bound}.
  By \cref{fact:sketch-and-solve} with $\varepsilon\le0.29$, 
\siamorarxiv{
  \begin{equation} \label{eq:sketch-and-solve-recap}
    \begin{split}
      \norm{\vec{x} - \vec{x}_0} &\le 2.82\sqrt{\varepsilon} \frac{\kappa}{\norm{\mat{A}}} \norm{\vec{r}(\vec{x})}, \\ \norm{\vec{R}(\vec{x} - \vec{x}_0)} &= \norm{\vec{r}(\vec{x}) - \vec{r}(\vec{x}_0)} \le 2.82\sqrt{\varepsilon} \norm{\vec{r}(\vec{x})}.
      \end{split}
    \end{equation}
}{
  \begin{equation} \label{eq:sketch-and-solve-recap}
    \begin{split}
      \norm{\vec{x} - \vec{x}_0} &\le 2.82\sqrt{\varepsilon} \frac{\kappa}{\norm{\mat{A}}} \norm{\vec{r}(\vec{x})}, \quad \norm{\vec{R}(\vec{x} - \vec{x}_0)} = \norm{\vec{r}(\vec{x}) - \vec{r}(\vec{x}_0)} \le 2.82\sqrt{\varepsilon} \norm{\vec{r}(\vec{x})}.
      \end{split}
    \end{equation}
}
  By the forward stability of multiplication by $\mat{S}$ \cref{eq:S_forward_stable} and the backward stability of least-squares solution by Householder \QR factorization \cite[Thm.~20.3]{Hig02}, we obtain
  \begin{equation*}
    \vec{\hat{x}}_0 = \argmin_{\vec{y} \in \real^n} \norm{(\mat{B} + \dta{B}')\vec{y} - (\vec{Sb} + \vec{\delta}')}, \quad \norm{\dta{B}'} \le \Cl{const:ss_backward} \norm{\mat{A}} u, \quad \norm{\vec{\delta}'} \le \Cr{const:ss_backward} \norm{\vec{b}} u.
  \end{equation*}
  Wedin's theorem (\cref{thm:wedin}) and our singular value estimates for $\mat{R}$ (\cref{fact:subspace_embedding}) then imply
  \begin{align*}
    \norm{\vec{\hat{x}}_0 - \vec{x}_0} &\le \Cl{const:ss_forward} \kappa u \left( \norm{\vec{x}_0} + \frac{\kappa}{\norm{\mat{A}}} \norm{\vec{r}(\vec{x}_0)} \right), \\
    \norm{\mat{R}(\vec{\hat{x}}_0 - \vec{x}_0)} &= \norm{\vec{r}(\vec{\hat{x}}_0) - \vec{r}(\vec{x}_0)} \le \Cr{const:ss_forward} u \left( \norm{\mat{A}} \norm{\vec{x}_0} + \kappa \norm{\vec{r}(\vec{x}_0)} \right).
  \end{align*}
  Apply the triangle inequality $\norm{\vec{x}_0} \le \norm{\vec{x}} + \norm{\vec{x} - \vec{x}_0}$, \cref{fact:sketch-and-solve}, and \cref{eq:sketch-and-solve-recap} to simplify this as
\siamorarxiv{
  \begin{equation} \label{eq:sketching-numerical}
    \begin{split}
      \norm{\vec{x}_0 - \vec{\hat{x}}_0} &\le \Cl{const:ss_forward_2} \kappa u \left( \norm{\vec{x}} + \frac{\kappa}{\norm{\mat{A}}} \norm{\vec{r}(\vec{x})} \right), \\ \norm{\mat{R}(\vec{x}_0 - \vec{\hat{x}}_0)} &\le \Cr{const:ss_forward_2} u \left( \norm{\mat{A}} \norm{\vec{x}} + \kappa \norm{\vec{r}(\vec{x})} \right).
      \end{split}
  \end{equation}
}{
  \begin{equation} \label{eq:sketching-numerical}
    \begin{split}
      \norm{\vec{x}_0 - \vec{\hat{x}}_0} &\le \Cl{const:ss_forward_2} \kappa u \left( \norm{\vec{x}} + \frac{\kappa}{\norm{\mat{A}}} \norm{\vec{r}(\vec{x})} \right), \quad \norm{\mat{R}(\vec{x}_0 - \vec{\hat{x}}_0)} \le \Cr{const:ss_forward_2} u \left( \norm{\mat{A}} \norm{\vec{x}} + \kappa \norm{\vec{r}(\vec{x})} \right).
      \end{split}
  \end{equation}
}
  Substituting \cref{eq:B_pert,eq:sketch-and-solve-recap,eq:sketching-numerical} into \cref{eq:initial-error-bound} yields the conclusion \cref{eq:initial-bound}.
\end{proof}

\subsection{Iterations for the error} \label{sec:iteration}

We now derive an iteration for the error $\vec{x} - \vec{\hat{x}}_i$ and the corresponding difference in residual $\vec{r}(\vec{x})-\vec{r}(\vec{\hat{x}}_i)$.
Consider a single loop iteration in \cref{alg:it-sk} and introduce quantities $\vec{e}_i,\vec{f}_i,\vec{g}_i,\vec{k}_i$ so that the following equations hold exactly
\begin{subequations} \label{eq:with_errors}
  \begin{align}
    \vec{\hat{r}}_i &= \vec{b} - \mat{A}\vec{\hat{x}}_i + \vec{f}_i, \\
    \vec{\hat{c}}_i &= \mat{A}^\top \vec{\hat{r}}_i + \vec{k}_i, \\
    \mat{\hat{R}}^\top \mat{\hat{R}} \vec{\hat{d}}_i &= \vec{\hat{c}}_i + \vec{g}_i, \label{eq:d} \\
    \vec{\hat{x}}_{i+1} &= \vec{\hat{x}}_i + \vec{\hat{d}}_i + \vec{e}_i.
  \end{align}
\end{subequations}
Combining these equations, we obtain
\begin{equation*} \label{eq:iteration2}
  \vec{\hat{x}}_{i+1} = \mat{\hat{R}}^{-1}\mat{G}\mat{\hat{R}}\vec{\hat{x}}_i + \mat{\hat{R}}^{-1}\mat{\hat{R}}^{-\top}\mat{A}^\top \vec{b} + \mat{\hat{R}}^{-1}\vec{h}_i, \quad \vec{h}_i \coloneqq \mat{\hat{R}}^{-\top}(\mat{A}^\top \vec{f}_i + \vec{g}_i + \vec{k}_i) + \vec{e}_i.
\end{equation*}
where the iteration matrix is
\begin{equation*}
  \mat{G} = \Id - \mat{\hat{R}}^{-\top} \mat{A}^\top\mat{A}\mat{\hat{R}}^{-1}.
\end{equation*}
Using the normal equations, we can easily check that
\begin{equation*} \label{eq:iteration}
  \vec{x} = \mat{\hat{R}}^{-1}\mat{G}\mat{\hat{R}}\vec{x} + \mat{\hat{R}}^{-1}\mat{\hat{R}}^{-\top} \mat{A}^\top \vec{b}.
\end{equation*}
Thus, we get the error
\begin{equation*}
  \vec{x} - \vec{\hat{x}}_{i+1} = \mat{\hat{R}}^{-1}\mat{G}\mat{\hat{R}}(\vec{x}-\vec{\hat{x}}_i) - \mat{\hat{R}}^{-1}\vec{h}_i.
\end{equation*}
The error over all $i$ iterations is thus
\begin{equation} \label{eq:iterated}
  \vec{x} - \vec{\hat{x}}_i = \mat{\hat{R}}^{-1}\mat{G}^i\mat{\hat{R}}(\vec{x}-\vec{\hat{x}}_0) - \sum_{j=1}^{i-1} \mat{\hat{R}}^{-1} \mat{G}^{i-1-j} \vec{h}_j.
\end{equation}
Taking norms in \cref{eq:iterated}, we see that
\begin{equation*}
  \norm{\vec{x} - \vec{\hat{x}}_i} \le \Big\|\mat{\hat{R}}^{-1}\Big\|\norm{\mat{G}}^i \big\|\mat{\hat{R}}(\vec{x}-\vec{\hat{x}}_0 )\big\| + \sum_{j=1}^{i-1} \Big\|\mat{\hat{R}}^{-1}\Big\| \norm{\mat{G}}^{i-1-j} \norm{\vec{h}_j}.
\end{equation*}
The norm of the iteration matrix $\mat{G}$ can now be estimated using \cref{prop:singular_values_fp}:
\begin{equation*}
  \norm{\mat{G}}= \norm{\Id - \mat{\hat{R}}^{-\top}\mat{A}^\top \mat{A}\mat{\hat{R}}^{-1}} \le g_{\rm IS} + \Cr{const:svfp} \kappa u \eqqcolon g.
\end{equation*}
In addition, by \cref{prop:singular_values_fp} and our assumption that $\Cr{const:c1}\kappa u < 1$ a constant $\Cr{const:c1}$ as large as we need, we obtain an estimate
\begin{equation*}
  \big\|\mat{\hat{R}}^{-1}\big\| = \frac{1}{\sigma_{\rm min}(\mat{\hat{R}})} \le 2\frac{\kappa}{\norm{\mat{A}}}.
\end{equation*}
Thus, we have the error bound
\begin{equation} \label{eq:error_bound}
  \norm{\vec{x}-\vec{\hat{x}}_i} \le 2\frac{\kappa}{\norm{\mat{A}}} \Big[ g^i \big\|\mat{\hat{R}}(\vec{x}-\vec{\hat{x}}_0)\big\| + \sum_{j=1}^{i-1} g^{i-1-j} \norm{\vec{h}_j} \Big].
\end{equation}
Proving the rapid convergence of iterative sketching just requires us to bound the norms of the $\vec{h}_i$.

We can also derive an iteration for the error in the residual.
Indeed, \cref{eq:iterated} to write
\begin{equation*}
  \vec{r}(\vec{x}) - \vec{r}(\vec{\hat{x}}_i) = \mat{A}(\vec{\hat{x}}_i - \vec{x}) = \mat{A}\mat{\hat{R}}^{-1}\mat{G}^i \mat{\hat{R}}(\vec{x}-\vec{\hat{x}}_0) + \sum_{j=1}^{i-1} \mat{A}\mat{\hat{R}}^{-1}\mat{G}^{i-1-j}\vec{h}_j.
\end{equation*}
Under our assumptions and with appropriate choice of $\Cr{const:c1}$, we have $\big\|\mat{A}\mat{\hat{R}}^{-1}\big\| \le 3$.
Thus,
\begin{equation} \label{eq:residual_bound}
  \norm{\vec{r}(\vec{\hat{x}}_i) - \vec{r}(\vec{x})} \le 3 g^i  \norm{\mat{\hat{R}}(\vec{\hat{x}}_0 - \vec{x})} + 3\sum_{j=1}^{i-1} g^{i-1-j} \norm{\vec{h_j}}.
\end{equation}

\subsection{Bounding the error quantities}

In this section, we derive bound on $\vec{h}_i$, beginning with its constituent parts $\vec{f}_i,\vec{k}_i,\vec{g}_i,\vec{e}_i$.
First, we use stability bounds for matrix--vector multiplication \cite[Eq.~(3.2)]{Hig02} to bound $\vec{f}_i$ as
\siamorarxiv{
\begin{equation} \label{eq:f_bound}
  \begin{split}
    \norm{\vec{f}_i} &\le \Cl{const:f} u (\norm{\mat{A}} \norm{\vec{\hat{x}}_i} + \norm{\vec{r}(\vec{\hat{x}}_i)}) \\&\le \Cr{const:f} u (\norm{\mat{A}} (\norm{\vec{\hat{x}}_i-\vec{x}} + \norm{\vec{x}}) + (\norm{\vec{r}(\vec{\hat{x}}_i)-\vec{r}(\vec{x})} + \norm{\vec{r}(\vec{x})})).
    \end{split}
\end{equation}
}{  
\begin{equation} \label{eq:f_bound}
  \begin{split}
    \norm{\vec{f}_i} &\le \Cl{const:f} u (\norm{\mat{A}} \norm{\vec{\hat{x}}_i} + \norm{\vec{r}(\vec{\hat{x}}_i)}) \le \Cr{const:f} u (\norm{\mat{A}} (\norm{\vec{\hat{x}}_i-\vec{x}} + \norm{\vec{x}}) + (\norm{\vec{r}(\vec{\hat{x}}_i)-\vec{r}(\vec{x})} + \norm{\vec{r}(\vec{x})})).
    \end{split}
\end{equation}  
}
Then, we use the matrix--vector stability bound again to obtain
\begin{equation} \label{eq:k_bound}
  \norm{\vec{k}_i} \le \Cl{const:k} u\norm{\mat{A}} (\norm{\vec{r}(\vec{\hat{x}}_i)} + \norm{\vec{f}_i}) \le \Cr{const:k} u\norm{\mat{A}} (\norm{\vec{r}(\vec{\hat{x}}_i) - \vec{r}(\vec{x})} + \norm{\vec{r}(\vec{x}} + \norm{\vec{f}_i}).
\end{equation}
Now, we seek to bound $\|\mat{\hat{R}}^{-\top}\vec{g}_i\|$.
By the backward stability of triangular solves \cite[Thm.~8.5]{Hig02}, there exist $\mat{E}_1,\mat{E}_2$ such that
\begin{equation} \label{eq:d_with_backward}
  (\mat{\hat{R}} + \mat{E}_1)^\top (\mat{\hat{R}} + \mat{E}_2)\vec{\hat{d}}_i = \vec{\hat{c}}_i,  \quad \norm{\mat{E}_1},\norm{\mat{E}_2} \le \C \norm{\mat{A}}u.
\end{equation}
Here, we also used \cref{prop:singular_values_fp}, which shows that $\|\mat{\hat{R}}\|$ is at most a small multiple of $\norm{\mat{A}}$.
Subtract from \cref{eq:with_errors} to obtain
\begin{equation*}
  (\mat{\hat{R}}^\top \mat{E}_2 + \mat{E}_1 \mat{R} + \mat{E}_1\mat{E}_2)\vec{\hat{d}}_i = \vec{g}_i.
\end{equation*}
Thus,
\begin{equation} \label{eq:Rg}
  \norm{\mat{\hat{R}}^{-\top}\vec{g}_i} = \norm{(\mat{E}_2 \mat{\hat{R}}^{-1} + \mat{\hat{R}}^{-\top}\mat{E}_1 + \mat{\hat{R}}^{-\top}\mat{E}_1 \mat{E}_2 \mat{\hat{R}}^{-1})\mat{\hat{R}}\vec{\hat{d}}_i} \le \C \kappa u \norm{\mat{\hat{R}}\vec{\hat{d}}_i}.
\end{equation}

Now, we bound $\|\mat{\hat{R}}\vec{\hat{d}}_i\|$.
Since $\mat{A}^\top \vec{r}(\vec{x}) = 0$, we can write
\begin{equation*}
  \mat{A}^\top \vec{\hat{r}}_i = \mat{A}^\top (\vec{\hat{r}}_i - \vec{r}(\vec{x})) = \mat{A}^\top (\vec{r}(\vec{x}_i)-\vec{r}(\vec{x})) + \mat{A}^\top \vec{f}_i.
\end{equation*}
Thus, we can write \cref{eq:d_with_backward} as 
\begin{equation*}
  \mat{\hat{R}}^{\top} (\Id + \mat{\hat{R}}^{-\top}\mat{E}_1) (\Id + \mat{E}_2\mat{\hat{R}}^{-1}) \mat{\hat{R}} \vec{\hat{d}}_i = \mat{A}^\top (\vec{r}(\vec{\hat{x}}_i)-\vec{r}(\vec{x})) + \mat{A}^\top \vec{f}_i + \vec{k}_i.
\end{equation*}
By \cref{prop:singular_values_fp}, we know that $\mat{\hat{R}}^{-\top}\mat{A}^{\top}$, $(\Id + \mat{R}^{-\top}\mat{E}_1)^{-1}$, and $(\Id + \mat{E}_2\mat{\hat{R}}^{-1})^{-1}$ all have norms which are at most $2$ assuming the constant $\Cr{const:c1}$ is large enough.
Thus,
\begin{equation} \label{eq:Rd}
  \norm{\mat{\hat{R}}\vec{\hat{d}}_i} \le \C \left(\norm{\vec{r}(\vec{\hat{x}}_i) - \vec{r}(\vec{x})} + \norm{\mat{\hat{R}}^{-\top} (\mat{A}^\top \vec{f}_i + \vec{k}_i)}\right).
\end{equation}
Using our bounds \cref{eq:f_bound,eq:k_bound} for $\vec{f}_i$ and $\vec{k}_i$, we obtain
\begin{equation} \label{eq:RAtfk}
  \norm{\mat{\hat{R}}^{-\top} (\mat{A}^\top \vec{f}_i + \vec{k}_i)} \le \C u\left(\norm{\mat{A}}\norm{\vec{\hat{x}}_i - \vec{x}} + \norm{\mat{A}}\norm{\vec{x}} + \kappa \norm{\vec{r}(\vec{\hat{x}}_i) - \vec{r}(\vec{x})} + \kappa\norm{\vec{r}}\right).
\end{equation}

Finally, we use the following simple bound for $\vec{e}_i$:
\begin{equation} \label{eq:e_bound}
  \norm{\vec{e}_i} \le \Cl{const:e} u \norm{\vec{\hat{x}}_{i+1}} \le \Cr{const:e} u (\norm{\vec{\hat{x}}_i - \vec{x}} + \norm{\vec{x}}).
\end{equation}

Bound the full error $\vec{h}_i$ as
\begin{equation*}
  \norm{\vec{h}_i} \le \norm{\mat{\hat{R}}^{-\top}(\mat{A}^\top \vec{f}_i + \vec{k}_i)} + \norm{\mat{\hat{R}}^{-\top}\vec{g}_i} + \big\|\mat{\hat{R}}\big\| \norm{\vec{e}_i}.
\end{equation*}
Invoke \cref{prop:singular_values_fp,eq:Rg,eq:Rd,eq:RAtfk,eq:e_bound} to obtain
\begin{equation} \label{eq:h_bound}
  \norm{\vec{h}_i} \le \C u \left[ \norm{\mat{A}} (\norm{\vec{\hat{x}}_i - \vec{x}} + \norm{\vec{x}}+\norm{\vec{\hat{x}}_{i+1}-\vec{x}}) + \kappa (\norm{\vec{r}(\vec{\hat{x}}_i)-\vec{r}(\vec{x})} + \norm{\vec{r}(\vec{x})}) \right].
\end{equation}

\subsection{Wrapping up}

Having bounded $\norm{\vec{h}_i}$, we can use \cref{eq:error_bound,eq:residual_bound} to bound the error and residual error, which we give abbreviated notations:
\begin{equation*}
  \err_i \coloneqq \norm{\vec{x} - \vec{\hat{x}}_i}, \quad \rerr_i \coloneqq \norm{\vec{r}(\vec{x}) - \vec{r}(\vec{\hat{x}}_i)}.
\end{equation*}
Also, set 
\begin{equation*}
  E_0 \coloneqq \big\|\mat{\hat{R}} (\vec{x}-\vec{\hat{x}}_0)\big\|.
\end{equation*}
Using these notations, \cref{eq:error_bound,eq:residual_bound} can be written as 
\begin{align*}
  \err_i &\le 3\kappa g^i \frac{E_0}{\norm{\mat{A}}} +  \Cl{const:iteration}\kappa u \sum_{j=1}^{i-1} g^{i-j-1} \left[ \err_j + \norm{\vec{x}} + \frac{\kappa}{\norm{\mat{A}}} (\rerr_j + \norm{\vec{r}(\vec{x})}) \right], \\
  \rerr_i &\le 4 g^i E_0 + \Cr{const:iteration}u\sum_{j=1}^{i-1-j} g^{i-j-1} \left[ \norm{\mat{A}}(\err_j + \norm{\vec{x}}) + \kappa(\rerr_j + \norm{\vec{r}(\vec{x})})  \right].
\end{align*}
The terms involving $\norm{\vec{x}}$ and $\norm{\vec{r}(\vec{x})}$ in the summations can be bounded by extending the summation to all natural numbers $j$ and using the formula for an infinite geometric series, leading to
\begin{align*}
  \err_i &\le 3\kappa g^i \frac{E_0}{\norm{\mat{A}}} \isiam{\!}+\isiam{\!} \Cl{const:iteration2}\kappa u \sum_{j=1}^{i-1} g^{i-j-1} \left[ \err_j + \frac{\kappa}{\norm{\mat{A}}} \rerr_j \right] + \Cr{const:iteration2} \kappa u \left[ \norm{\vec{x}} \isiam{\!}+\isiam{\!} \frac{\kappa}{\norm{\mat{A}}} \norm{\vec{r}(\vec{x})} \right], \\
  \rerr_i &\le 4 g^i E_0 \isiam{\!}+\isiam{\!} \Cr{const:iteration2}u\sum_{j=1}^{i-1-j} g^{i-j-1} \left[ \norm{\mat{A}}\err_j + \kappa\rerr_j)  \right] \isiam{\!}+\isiam{\!} \Cr{const:iteration2} u\left[ \norm{\mat{A}}\norm{\vec{x}} + \kappa \norm{\vec{r}(\vec{x})} \right].
\end{align*}

To solve this recurrence, define
\begin{equation*}
  t_i \coloneqq \err_i + \frac{\kappa}{\norm{\mat{A}}} \rerr_i, \quad v \coloneqq 7\kappa \frac{E_0}{\norm{\mat{A}}}, \quad q \coloneqq 2\Cr{const:iteration2}\kappa u, \quad w \coloneqq q \left[ \norm{\vec{x}} + \frac{\kappa}{\norm{\mat{A}}} \norm{\vec{r}(\vec{x})} \right].
\end{equation*}
Then, we have
\begin{equation} \label{eq:recurrence}
  t_i \le v g^i + q\sum_{j=1}^{i-1} g^{i-j-1} t_j + w.
\end{equation}
With some effort, one can derive a the following bound for the solution of this recurrence:

\begin{proposition}[Solution to recurrence] \label{prop:recurrence}
  Suppose that $t_0,t_1,t_2,\ldots$ are real numbers satisfying the recurrence \cref{eq:recurrence} and suppose that $g,q,g+q \in (0,1)$.
  Then,
  \begin{equation*}
    t_i \le v(g+q)^i + \frac{1-g}{1-(g+q)} w \quad \text{for } i =0,1,2,\ldots.
  \end{equation*}
\end{proposition}

The proof is a simple inductive argument and is omitted.
Plugging in this bound, unpacking the definition of all the symbols, and using the bound \cref{lem:initial-error} on the initial error $E_0$ establishes the two bounds in \cref{thm:it_sk_forward}. \hfill $\proofbox$

\section{Conclusions} \label{sec:conclusions}

This paper is to first to establish that a fast randomized least-squares solver---iterative sketching, in our case---is forward stable, producing solutions of comparable forward error to Householder \QR.
We complement our stability analysis with numerical experiments, demonstrating that iterative sketching can lead to speedups over Householder \QR factorization for large, highly overdetermined least-squares problems.
We expect that our analysis will carry over straightforwardly to versions of iterative sketching with damping and momentum (\cref{sec:damp-mom}).

We hope this work will lead to further investigation into the numerical stability of randomized algorithms.
In particular, our work raises the following questions:
\begin{itemize}
\item Is sketch-and-precondition with the sketch-and-solve initialization forward stable?
\item Is there a fast randomized least-squares solver which is \emph{backward stable}?
\end{itemize}


\section*{Disclaimer}
This report was prepared as an account of work sponsored by an agency of the United States Government. Neither the United States Government nor any agency thereof, nor any of their employees, makes any warranty, express or implied, or assumes any legal liability or responsibility for the accuracy, completeness, or usefulness of any information, apparatus, product, or process disclosed, or represents that its use would not infringe privately owned rights. Reference herein to any specific commercial product, process, or service by trade name, trademark, manufacturer, or otherwise does not necessarily constitute or imply its endorsement, recommendation, or favoring by the United States Government or any agency thereof. The views and opinions of authors expressed herein do not necessarily state or reflect those of the United States Government or any agency thereof.

\bibliographystyle{halpha}
\bibliography{refs}

\newcommand{\etalchar}[1]{$^{#1}$}
\begin{thebibliography}{MNTW24}
\expandafter\ifx\csname url\endcsname\relax
  \def\url#1{\texttt{#1}}\fi
\expandafter\ifx\csname doi\endcsname\relax
  \def\doi#1{\burlalt{doi:#1}{http://dx.doi.org/#1}}\fi
\expandafter\ifx\csname urlprefix\endcsname\relax\def\urlprefix{}\fi
\expandafter\ifx\csname href\endcsname\relax
  \def\href#1#2{#2}\fi
\expandafter\ifx\csname burlalt\endcsname\relax
  \def\burlalt#1#2{\href{#2}{#1}}\fi

\bibitem[AMT10]{AMT10}
Haim Avron, Petar Maymounkov, and Sivan Toledo.
\newblock Blendenpik: {{Supercharging LAPACK}}'s least-squares solver.
\newblock {\em SIAM Journal on Scientific Computing}, 32(3):1217--1236, January
  2010.
\newblock \doi{10.1137/090767911}.

\bibitem[Bha97]{Bha97}
Rajendra Bhatia.
\newblock {\em Matrix Analysis}, volume 169 of {\em Graduate {{Texts}} in
  {{Mathematics}}}.
\newblock {Springer}, New York, 1997.
\newblock \doi{10.1007/978-1-4612-0653-8}.

\bibitem[Bj{\"o}87]{Bjo87}
{\AA}ke Bj{\"o}rck.
\newblock Stability analysis of the method of seminormal equations for linear
  least squares problems.
\newblock {\em Linear Algebra and its Applications}, 88--89:31--48, April 1987.
\newblock \doi{10.1016/0024-3795(87)90101-7}.

\bibitem[Bj{\"o}96]{Bjo96}
{\AA}ke Bj{\"o}rck.
\newblock {\em Numerical Methods for Least Squares Problems}.
\newblock {SIAM}, {Philadelphia}, 1996.
\newblock \doi{10.1137/1.9781611971484}.

\bibitem[BSW14]{BSW14}
Pierre~Baldi Baldi, Peter Sadowski, and Daniel Whiteson.
\newblock Searching for exotic particles in high-energy physics with deep
  learning.
\newblock {\em Nature Communications}, 5(1):4308, July 2014.
\newblock \doi{10.1038/ncomms5308}.

\bibitem[CDD{\etalchar{+}}23]{CDD+23}
Younghyun Cho, James~W. Demmel, Micha{\l} Derezi{\'n}ski, Haoyun Li, Hengrui
  Luo, Michael~W. Mahoney, and Riley~J. Murray.
\newblock Surrogate-based {{Autotuning}} for {{Randomized Sketching
  Algorithms}} in {{Regression Problems}}, August 2023.
\newblock \urlprefix\url{https://arxiv.org/abs/2308.15720v1}.

\bibitem[Coh16]{Coh16}
Michael~B. Cohen.
\newblock Nearly tight oblivious subspace embeddings by trace inequalities.
\newblock In {\em Proceedings of the {{Twenty-Seventh Annual ACM-SIAM
  Symposium}} on {{Discrete Algorithms}}}, pages 278--287. {SIAM}, January
  2016.
\newblock \doi{10.1137/1.9781611974331.ch21}.

\bibitem[DEF{\etalchar{+}}23]{DEF+23}
Mateo D{\'i}az, Ethan~N. Epperly, Zachary Frangella, Joel~A. Tropp, and
  Robert~J. Webber.
\newblock Robust, randomized preconditioning for kernel ridge regression, April
  2023.
\newblock \urlprefix\url{https://arxiv.org/abs/2304.12465v3}.

\bibitem[DM23]{DM23}
Yijun Dong and Per-Gunnar Martinsson.
\newblock Simpler is better: a comparative study of randomized pivoting
  algorithms for {CUR} and interpolative decompositions.
\newblock {\em Advances in Computational Mathematics}, 49(4):66, August 2023.
\newblock \doi{10.1007/s10444-023-10061-z}.

\bibitem[GVL13]{GV13}
Gene~H. Golub and Charles~F. Van~Loan.
\newblock {\em Matrix Computations}.
\newblock {Johns Hopkins Press}, {Baltimore}, fourth edition, 2013.
\newblock \doi{10.1145/2842602}.

\bibitem[Hig02]{Hig02}
Nicholas~J. Higham.
\newblock {\em Accuracy and Stability of Numerical Algorithms}.
\newblock {SIAM}, Philadelphia, 2002.
\newblock \doi{doi.org/10.1137/1.9780898718027}.

\bibitem[KT24]{KT24}
Anastasia Kireeva and Joel~A. Tropp.
\newblock Randomized matrix computations: {{Themes}} and variations, February
  2024.
\newblock \urlprefix\url{https://arxiv.org/abs/2402.17873v2}.

\bibitem[LP21]{LP21}
Jonathan Lacotte and Mert Pilanci.
\newblock Faster least squares optimization, April 2021.
\newblock \urlprefix\url{https://arxiv.org/abs/1911.02675v3}.

\bibitem[MDM{\etalchar{+}}22]{MDM+22}
Riley Murray, James Demmel, Michael~W. Mahoney, N.~Benjamin Erichson, Maksim
  Melnichenko, Osman~Asif Malik, Laura Grigori, Micha{\l} Derezi{\'n}ski,
  Miles~E. Lopes, Tianyu Liang, and Hengrui Luo.
\newblock Randomized numerical linear algebra: {{A}} perspective on the field
  with an eye to software, November 2022.
\newblock \urlprefix\url{https://arxiv.org/abs/2302.11474v2}.

\bibitem[MNTW24]{MNTW23}
Maike Meier, Yuji Nakatsukasa, Alex Townsend, and Marcus Webb.
\newblock Are sketch-and-precondition least squares solvers numerically stable?
\newblock {\em SIAM Journal on Matrix Analysis and Applications}, 2024.
\newblock To appear (preprint available at
  \url{https://arxiv.org/abs/2302.07202v2}).

\bibitem[MT20]{MT20a}
Per-Gunnar Martinsson and Joel~A. Tropp.
\newblock Randomized numerical linear algebra: {{Foundations}} and algorithms.
\newblock {\em Acta Numerica}, 29:403--572, May 2020.
\newblock \doi{10.1017/S0962492920000021}.

\bibitem[OPA19]{OPA19}
Ibrahim~Kurban Ozaslan, Mert Pilanci, and Orhan Arikan.
\newblock Iterative {{Hessian}} sketch with momentum.
\newblock In {\em 2019 IEEE International Conference on Acoustics, Speech and
  Signal Processing}, pages 7470--7474. {IEEE}, 2019.
\newblock \doi{10.1109/ICASSP.2019.8682720}.

\bibitem[PS82]{PS82}
Christopher~C. Paige and Michael~A. Saunders.
\newblock {{LSQR}}: {{An}} algorithm for sparse linear equations and sparse
  least squares.
\newblock {\em ACM Transactions on Mathematical Software}, 8(1):43--71, March
  1982.
\newblock \doi{10.1145/355984.355989}.

\bibitem[PW16]{PW16a}
Mert Pilanci and Martin~J. Wainwright.
\newblock Iterative {{Hessian}} sketch: {{Fast}} and accurate solution
  approximation for constrained least-squares.
\newblock {\em Journal of Machine Learning Research}, 17(1):1842--1879, 2016.
\newblock \urlprefix\url{http://jmlr.org/papers/v17/14-460.html}.

\bibitem[RCR17]{RCR17}
Alessandro Rudi, Luigi Carratino, and Lorenzo Rosasco.
\newblock {{FALKON}}: {{An}} optimal large scale kernel method.
\newblock In {\em Advances in {{Neural Information Processing Systems}}},
  volume~30, 2017.
\newblock
  \urlprefix\url{https://proceedings.neurips.cc/paper/2017/hash/05546b0e38ab9175cd905eebcc6ebb76-Abstract.html}.

\bibitem[RT08]{RT08}
Vladimir Rokhlin and Mark Tygert.
\newblock A fast randomized algorithm for overdetermined linear least-squares
  regression.
\newblock {\em Proceedings of the National Academy of Sciences},
  105(36):13212--13217, September 2008.
\newblock \doi{10.1073/pnas.0804869105}.

\bibitem[Sar06]{Sar06}
Tamas Sarlos.
\newblock Improved approximation algorithms for large matrices via random
  projections.
\newblock In {\em 2006 47th {{Annual IEEE Symposium}} on {{Foundations}} of
  {{Computer Science}}}, pages 143--152. {IEEE}, October 2006.
\newblock \doi{10.1109/FOCS.2006.37}.

\bibitem[SB00]{SB00}
Alex~J. Smola and Peter~L. Bartlett.
\newblock Sparse greedy {{Gaussian}} process regression.
\newblock In {\em Advances in Neural Information Processing Systems},
  volume~13, 2000.
\newblock
  \urlprefix\url{https://proceedings.neurips.cc/paper/2000/hash/3214a6d842cc69597f9edf26df552e43-Abstract.html}.

\bibitem[TYUC19]{TYUC19}
Joel~A. Tropp, Alp Yurtsever, Madeleine Udell, and Volkan Cevher.
\newblock Streaming low-rank matrix approximation with an application to
  scientific simulation.
\newblock {\em SIAM Journal on Scientific Computing}, 41(4):A2430--A2463,
  January 2019.
\newblock \doi{10.1137/18M1201068}.

\bibitem[Wed73]{Wed73}
Per-{\AA}ke Wedin.
\newblock Perturbation theory for pseudo-inverses.
\newblock {\em BIT Numerical Mathematics}, 13(2):217--232, June 1973.
\newblock \doi{10.1007/BF01933494}.

\end{thebibliography}

\appendix 
\newpage

\section{Proof of \texorpdfstring{\cref{fact:sketch-and-solve}}{Fact 4}} \label{app:proofs}
  The first bound $\norm{\vec{r}(\vec{x}_0)} \le (1+\varepsilon)/(1-\varepsilon) \cdot \norm{\vec{r}(\vec{x})}$ is standard \cite[Prop.~5.3]{KT24}.
  We now prove the other two bounds.
  For the optimal solution $\vec{x}$ to the least-squares problem \cref{eq:ls}, the residual $\vec{r}(\vec{x})$ is orthogonal to the range of $\mat{A}$.
  Thus, by the Pythagorean theorem,
  \begin{equation*}
    \begin{split}
      \norm{\vec{r}(\vec{x}_0)} & = \norm{\vec{b} - \mat{A}\vec{x}}^2 + \norm{\mat{A}(\vec{x} - \vec{x}_0)}^2 = \norm{\vec{r}(\vec{x})}^2 + \norm{\vec{r}(\vec{x}) - \vec{r}(\vec{x}_0)}^2.
    \end{split}
  \end{equation*}
  Thus,
  \begin{equation*}
    \norm{\vec{r}(\vec{x}) - \vec{r}(\vec{x}_0)} = \sqrt{\norm{\vec{r}(\vec{x}_0)}^2 - \norm{\vec{r}(\vec{x})}^2} \le \sqrt{\left( \frac{1+\varepsilon}{1-\varepsilon} \right)^2 - 1} \cdot \norm{\vec{r}(\vec{x})} = \frac{2\sqrt{\varepsilon}}{1-\varepsilon} \norm{\vec{r}(\vec{x})}.
  \end{equation*}
  Finally, we obtain
  \begin{equation*}
    \norm{\vec{x} - \vec{x}_0} \le \frac{1}{\sigma_{\rm min}(\mat{A})} \norm{\mat{A}(\vec{x} - \vec{x}_0)} = \frac{1}{\sigma_{\rm min}(\mat{A})} \norm{\vec{r}(\vec{x}) - \vec{r}(\vec{x}_0)} \le \frac{2\sqrt{\varepsilon}}{1-\varepsilon} \frac{\kappa}{\norm{\mat{A}}}\norm{\vec{r}(\vec{x})}.
  \end{equation*}
  This completes the proof. \hfill $\proofbox$

\section{Damping and momentum} \label{sec:damp-mom}
The convergence rate of iterative sketching can be improved by incorporating \emph{damping} and \emph{momentum} \siamorarxiv{\cite{LP21,OPA19}.}{\cite{OPA19,LP21}.}
In our experience, with optimized parameter choices, damping and momentum accelerate iterative sketching by a modest amount (e.g., runtime reduced by a factor of two).
There may be settings (e.g., when memory is limited) where the advantages of these schemes is more pronounced.

Iterative sketching with damping and momentum use an iteration of the form:
\begin{equation} \label{eq:damping_momentum}
  (\mat{S}\mat{A})^\top (\mat{S}\mat{A}) \vec{d}_i = \mat{A}^\top (\vec{b} - \mat{A}\vec{x}_i), \quad \vec{x}_{i+1} \coloneqq \vec{x}_i + \alpha \vec{d}_i + \beta (\vec{x}_i - \vec{x}_{i-1}).
\end{equation}
Let $\mat{S}$ be a subspace embedding for $\range(\mat{A})$ with distortion $\varepsilon$.
We consider two schemes:
\begin{itemize}
\item \textbf{Iterative sketching with damping:} Use parameters 
  \begin{equation} \label{eq:damping}
    \alpha = \frac{(1-\varepsilon^2)^2}{1+\varepsilon^2}, \quad \beta = 0.
  \end{equation}
\item \textbf{Iterative sketching with momentum:} Use parameters
  \begin{equation} \label{eq:momentum}
    \alpha = (1-\varepsilon^2)^2, \quad \beta = \varepsilon^2.
  \end{equation}
\end{itemize}
In \siamorarxiv{\cite{LP21,OPA19}}{\cite{OPA19,LP21}}, it is shown that these parameter choices optimize the rate of convergence.
Iterative sketching with momentum is similar to the Chebyshev semi-iterative method \cite[\S11.2.8]{GV13}, except that the $\alpha$ and $\beta$ parameters remain constant during the iteration.
For both methods, we choose $\vec{x}_0$ to be the sketch-and-solve solution \cref{eq:sketched_ls}.
For a stable implementation, we replace line 9 in \cref{alg:it-sk} by \cref{eq:damping_momentum}.

\begin{theorem}[Damping and momentum] \label{thm:damp_mom}
  Let $\mat{S}$ be a subspace embedding with distortion $\varepsilon \in (0,1)$.
  The iterates $\vec{x}_0,\vec{x}_1,\ldots$ of iterative sketching with damping satisfy
  \begin{equation} \label{eq:damping_momentum_error}
    \norm{\vec{x} - \vec{x}_i} \le C \kappa \,g_{\rm damp}^i\, \frac{\norm{\vec{r}(\vec{x})}}{\norm{\mat{A}}}, \quad \norm{\vec{r}(\vec{x}) - \vec{r}(\vec{x}_i)} \le C \,g_{\rm damp}^i \,\norm{\vec{r}(\vec{x})},
  \end{equation}
  where the convergence rate $g_{\rm damp}$ and prefactor constant $C$ are 
  \begin{equation*}
    g_{\rm damp} = \frac{2 \varepsilon}{1+\varepsilon^2}, \quad C = \frac{2(1+\varepsilon)\sqrt{\varepsilon}}{(1-\varepsilon)^2}.
  \end{equation*}
  For $i\ge 2$, the iterates of iterative sketching with momentum satisfy
  \begin{equation} \label{eq:momentum_error}
    \norm{\vec{x} - \vec{x}_i} \le C'(i-1)\kappa\, g_{\rm mom}^i \frac{\norm{\vec{r}(\vec{x})}}{\norm{\mat{A}}},\: \norm{\vec{r}(\vec{x}) - \vec{r}(\vec{x}_i)} \le C'(i-1) \, g_{\rm mom}^i \norm{\vec{r}(\vec{x})}.
  \end{equation}
  where the convergence rate $g_{\rm mom}$ and prefactor constant $C'$ are
  \begin{equation*}
    g_{\rm mom} = \varepsilon, \quad C' = \frac{8\sqrt{2}(1+\varepsilon)}{(1-\varepsilon)^2\sqrt{\varepsilon}}.
  \end{equation*}
\end{theorem}

\ifsiam \else
\begin{proof} The proof of the bound \cref{eq:damping_momentum_error} concerning the method with damping is identifical to \cref{thm:it-sk-convergence} and is omitted.
  We prove \cref{eq:momentum_error}.
  The error $\vec{x} - \vec{x}_i$ obeys the following recurrence \cite[\S2.2]{OPA19}:
  \iarxiv{\begin{equation*}
    \twobyone{\vec{x} - \vec{x}_{i+1}}{\vec{x}-\vec{x}_i} = (\Id_2 \otimes \mat{R}^{-1})\mat{T}(\Id_2 \otimes \mat{R}) \twobyone{\vec{x}-\vec{x}_i}{\vec{x}-\vec{x}_{i-1}}, \quad \mat{T} = \twobytwo{(1+\beta)\Id - \alpha \mat{R}^{-\top}\mat{A}^\top \mat{A}\mat{R}^{-1}}{-\beta\Id}{\Id}{\mat{0}}.
  \end{equation*}
  Here, $\Id_2$ denotes the $2\times 2$ identity matrix, $\Id$ denotes the $n\times n$ matrix, and $\otimes$ denotes Kronecker product.
}
\isiam{
  \begin{equation*}
    \twobyone{\vec{x} - \vec{x}_{i+1}}{\vec{x}-\vec{x}_i} = (\Id_2 \otimes \mat{R}^{-1})\mat{T}(\Id_2 \otimes \mat{R}) \twobyone{\vec{x}-\vec{x}_i}{\vec{x}-\vec{x}_{i-1}}
  \end{equation*}
  where $\Id_2$ denotes the $2\times 2$ identity matrix, $\Id$ denotes the $n\times n$ matrix, $\otimes$ denotes Kronecker product, and
  \begin{equation*}
    \mat{T} = \twobytwo{(1+\beta)\Id - \alpha \mat{R}^{-\top}\mat{A}^\top \mat{A}\mat{R}^{-1}}{-\beta\Id}{\Id}{\mat{0}}.
  \end{equation*}
}
  Consequently,
\iarxiv{
  \begin{equation*}
    \twobyone{\mat{R}(\vec{x} - \vec{x}_{i+1})}{\mat{R}(\vec{x}-\vec{x}_i)} = \mat{T} \twobyone{\mat{R}(\vec{x}-\vec{x}_i)}{\mat{R}(\vec{x}-\vec{x}_{i-1})} \text{ for } i=1,2,\ldots \implies \twobyone{\mat{R}(\vec{x} - \vec{x}_i)}{\mat{R}(\vec{x}-\vec{x}_{i-1})} = \mat{T}^i \twobyone{\mat{R}(\vec{x} - \vec{x}_{0})}{\mat{R}(\vec{x}-\vec{x}_{-1})}.
  \end{equation*}
}
\isiam{
  \begin{equation*}
    \twobyone{\mat{R}(\vec{x} - \vec{x}_{i+1})}{\mat{R}(\vec{x}-\vec{x}_i)} = \mat{T} \twobyone{\mat{R}(\vec{x}-\vec{x}_i)}{\mat{R}(\vec{x}-\vec{x}_{i-1})} \text{ for } i=1,2,\ldots
  \end{equation*}
  so
  \begin{equation*}
    \twobyone{\mat{R}(\vec{x} - \vec{x}_i)}{\mat{R}(\vec{x}-\vec{x}_{i-1})} = \mat{T}^i \twobyone{\mat{R}(\vec{x} - \vec{x}_{0})}{\mat{R}(\vec{x}-\vec{x}_{-1})}.
  \end{equation*}
}
  Therefore,
  \begin{equation} \label{eq:momentum-res-bound}
    \norm{\mat{R}(\vec{x}-\vec{x}_i)} \le \norm{\twobyone{\mat{R}(\vec{x} - \vec{x}_i)}{\mat{R}(\vec{x}-\vec{x}_{i-1})}} \le \norm{\mat{T}^i} \norm{\twobyone{\mat{R}(\vec{x} - \vec{x}_{0})}{\mat{R}(\vec{x}-\vec{x}_{-1})}}.
  \end{equation}
  By a computation similar to \cref{eq:sketch-and-solve-initial-bound}, we have
\iarxiv{
  \begin{equation} \label{eq:momentum-initial}
    \norm{\mat{R}(\vec{x} - \vec{x}_{0})} = \norm{\mat{R}(\vec{x} - \vec{x}_{-1})} \le 2\frac{1+\varepsilon}{1-\varepsilon}\sqrt{\varepsilon} \norm{\vec{r}(\vec{x})} \implies \norm{\twobyone{\mat{R}(\vec{x} - \vec{x}_{0})}{\mat{R}(\vec{x}-\vec{x}_{-1})}} \le 2\sqrt{2}\frac{1+\varepsilon}{1-\varepsilon}\sqrt{\varepsilon} \norm{\vec{r}(\vec{x})}.
  \end{equation}
}
\isiam{
  \begin{equation*}
    \norm{\mat{R}(\vec{x} - \vec{x}_{0})} = \norm{\mat{R}(\vec{x} - \vec{x}_{-1})} \le 2\frac{1+\varepsilon}{1-\varepsilon}\sqrt{\varepsilon} \norm{\vec{r}(\vec{x})}
  \end{equation*}
  so
  \begin{equation} \label{eq:momentum-initial}
     \norm{\twobyone{\mat{R}(\vec{x} - \vec{x}_{0})}{\mat{R}(\vec{x}-\vec{x}_{-1})}} \le 2\sqrt{2}\frac{1+\varepsilon}{1-\varepsilon}\sqrt{\varepsilon} \norm{\vec{r}(\vec{x})}.
  \end{equation}
}
  All we need to instantiate \cref{eq:momentum-res-bound} is a bound on $\norm{\mat{T}^i}$, which we now proceed to discover.
  
  Consider the Hermitian eigendecomposition
  \begin{equation*}
    (1+\beta)\Id - \alpha \mat{R}^{-\top}\mat{A}^\top \mat{A}\mat{R}^{-1} = \mat{V}\diag(\lambda_j : j=1,2,\ldots,n)\mat{V}^*,
  \end{equation*}
  where $\mat{V}$ is unitary and $\lambda_1,\ldots,\lambda_n$ are real.
  Using the singular value bounds \cref{fact:subspace_embedding} and the values of $\alpha$ and $\beta$, we obtain bounds
  \begin{equation} \label{eq:lambda_bounds}
    |\lambda_j| \le 2\varepsilon \quad \text{for } j=1,2,\ldots,n.
  \end{equation}
  Using this eigendecomposition, $\mat{T}$ can be block diagonalized as
  \begin{equation} \label{eq:T_block_diag}
    \mat{T} = \mat{U} \diag\left(\mat{L}_j : j=1,2,\ldots,n\right) \mat{U}^*, \quad \mat{U} \coloneqq \mat{\Pi}(\Id_2 \otimes \mat{V}), \quad \mat{L}_j \coloneqq \twobytwo{\lambda_j}{-\beta}{1}{0}
  \end{equation}
  where $\mat{\Pi}$ is a perfect shuffle permutation \cite[\S1.2.11]{GV13}.
  Using an eigendecomposition of $\mat{L}_j$, we obtain an exact formula for the powers of $\mat{L}_j$:
\iarxiv{
  \begin{equation*}
    \mat{L}_j^i = \frac{1}{\rho^+_j - \rho^-_j} \twobytwo{ (\rho^+_j)^{i+1} - (\rho^-_j)^{i+1} }{ \beta \big[(\rho^+_j)^{i} - (\rho^-_j)^{i}\big] }{ (\rho^+_j)^{i} - (\rho^-_j)^{i} }{ -\rho^+_j \rho^-_j \big[(\rho^+_j)^{i-1} - (\rho^-_j)^{i-1}\big] },\quad \rho^\pm_j \coloneqq \frac{\lambda_j \pm \sqrt{\lambda_j^2-4\beta}}{2}.
  \end{equation*}
}
\isiam{
  \begin{equation*}
    \mat{L}_j^i = \frac{1}{\rho^+_j - \rho^-_j} \twobytwo{ (\rho^+_j)^{i+1} - (\rho^-_j)^{i+1} }{ \beta \big[(\rho^+_j)^{i} - (\rho^-_j)^{i}\big] }{ (\rho^+_j)^{i} - (\rho^-_j)^{i} }{ -\rho^+_j \rho^-_j \big[(\rho^+_j)^{i-1} - (\rho^-_j)^{i-1}\big] },
  \end{equation*}
  where
  \begin{equation*}
    \rho^\pm_j \coloneqq \frac{\lambda_j \pm \sqrt{\lambda_j^2-4\beta}}{2}.
  \end{equation*}
}
  In the degenerate case $\rho^+_j = \rho^-_j$, this formula should be interpreted as the limit of this expression as $\rho^-_j \to \rho^+_j$.
  Using the bounds \cref{eq:lambda_bounds} on $\lambda_j$, we see that $\sqrt{\lambda_j^2-4\beta}$ is purely imaginary and thus
  \begin{equation*}
    \left| \rho^{\pm}_j \right| = \frac{1}{2} \sqrt{ \lambda_j^2 + \left(4\beta - \lambda_j^2\right)} = \sqrt{\beta} = \varepsilon.
  \end{equation*}
  Using this bound on $|\rho^\pm_j|$, we can bound the norm of $\mat{L}_j^i$ by the absolute sum of its entries, obtaining
  \begin{equation*}
    \norm{\mat{L}_j^i} \le (i-1)(1+\varepsilon^2+2\varepsilon)\varepsilon^{i-1} \le 4(i-1)\varepsilon^{i-1} \quad \text{for } i \ge 2.
  \end{equation*}
  Thus, substituting into \cref{eq:T_block_diag}, we see that $\norm{\mat{T}^i} \le 4(i-1)\varepsilon^{i-1}$.
  Substituting into \cref{eq:momentum-res-bound} and using \cref{eq:momentum-initial}, we obtain
  \begin{equation*}
    \norm{\mat{R}(\vec{x}-\vec{x}_i)} \le 8\sqrt{2} \frac{1+\varepsilon}{1-\varepsilon} \varepsilon^{-1/2} \cdot (i-1)\varepsilon^i \norm{\vec{r}(\vec{x})}.
  \end{equation*}
  The rest of the proof is identical to that of \cref{thm:it-sk-convergence}. \end{proof}
\fi

Damping and momentum meaningfully improve the rate of convergence of iterative sketching.
In particular, the optimally damped \cref{eq:damping} and the optimal momentum \cref{eq:momentum} scheme are geometrically convergent for a subspace dimension with any distortion $\varepsilon < 1$.
In particular, we can use a smaller embedding dimension (say, $d\approx 4n$) while maintaining a convergent scheme.

Choose the damping and momentum parameters, one needs an estimate of the distortion $\varepsilon$ of the embedding $\mat{S}$.
This is a disadvantage of iterative sketching with damping and momentum compared to sketch-and-precondition, which does not have any tunable parameters.
In practice, we recommend using the value $\varepsilon = \sqrt{n/d}$ (or a modest multiple thereof) in the expressions \cref{eq:damping,eq:momentum} for the damping and momentum parameters.
Our recommended setting for $d$, analagous to \cref{eq:dimension}, is
\begin{equation*}
  d \coloneqq \max\left(\left\lceil an \cdot \exp\left(\mathrm{W}\left(\frac{4m}{an^2} \log \left(\frac{1}{u}\right) \right)\right)\right\rceil, 4n\right)
\end{equation*}
with $a = 1$ for iterative sketching with momentum and $a = 2$ for iterative sketching with damping.

\begin{figure}[t]
  \centering
  
  \includegraphics[width=0.32\textwidth]{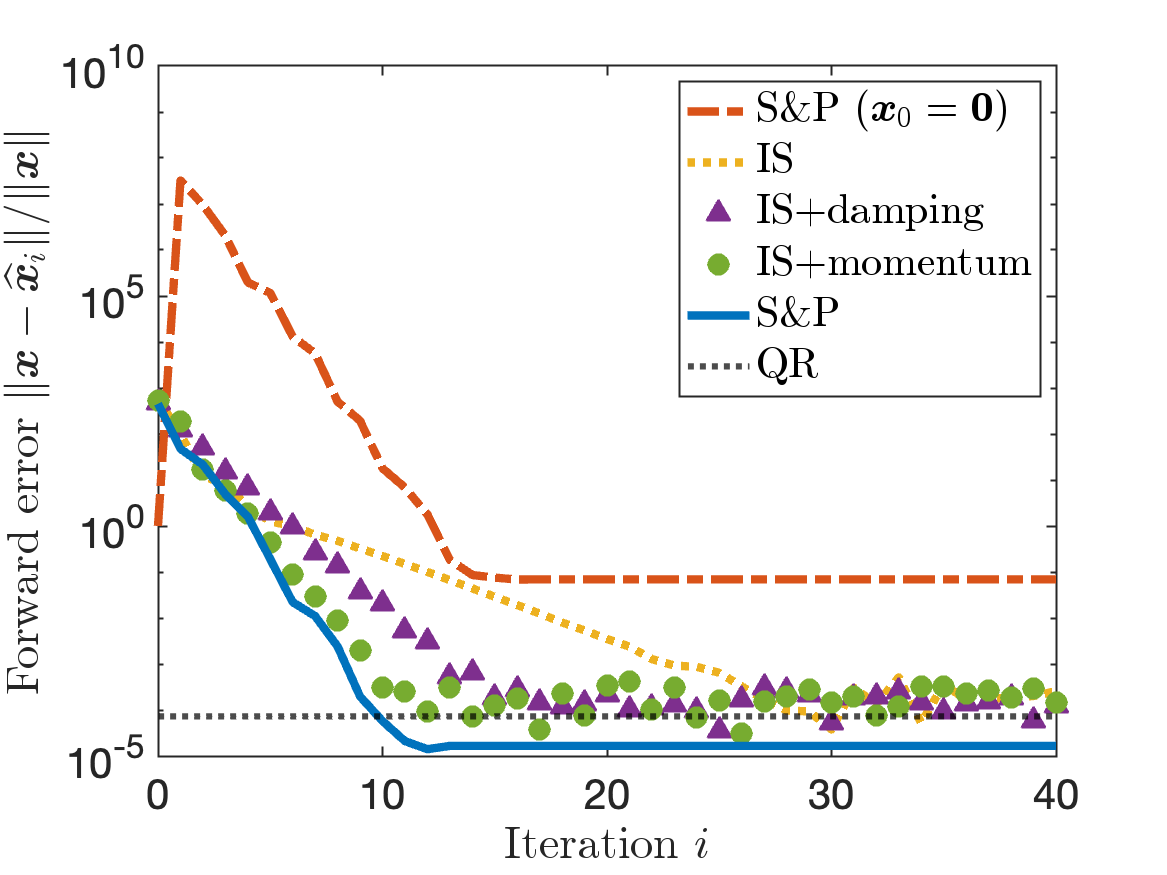}
  \hfill
  \includegraphics[width=0.32\textwidth]{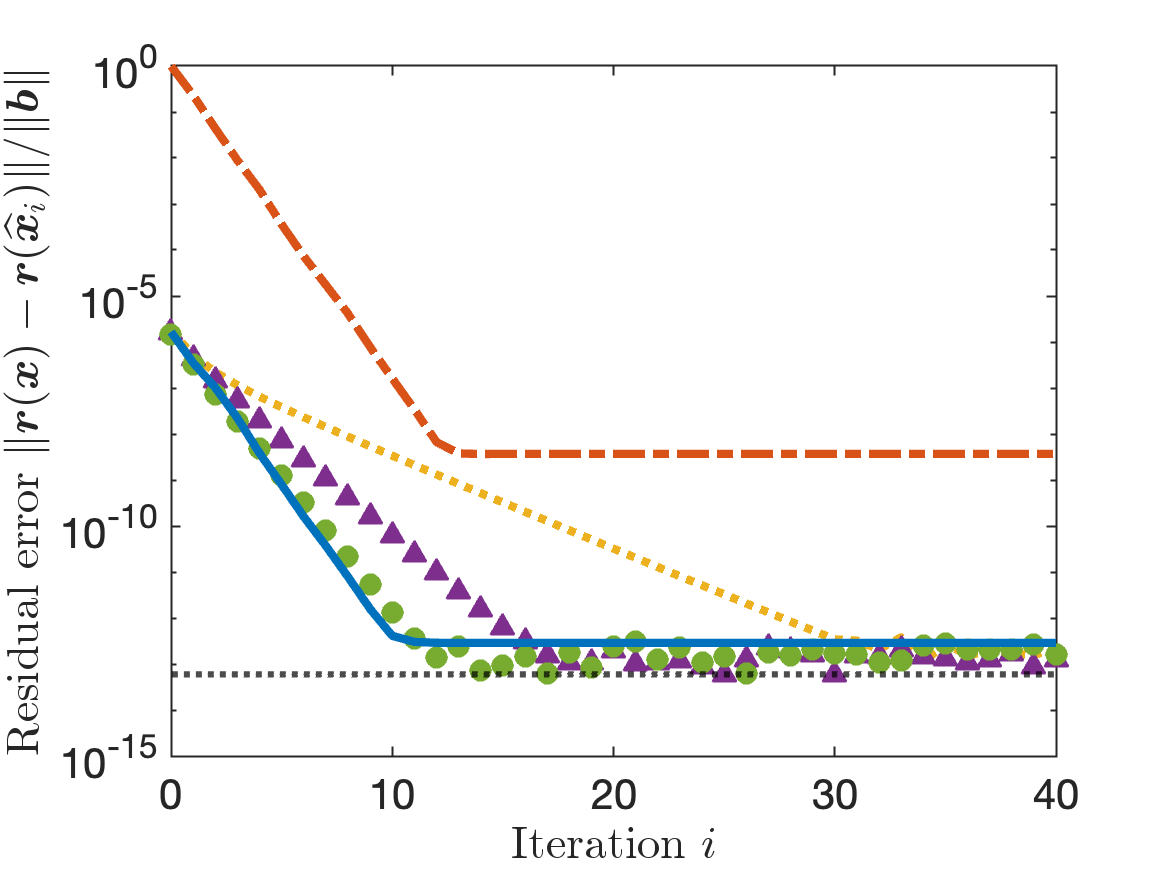}
  \hfill
  \includegraphics[width=0.32\textwidth]{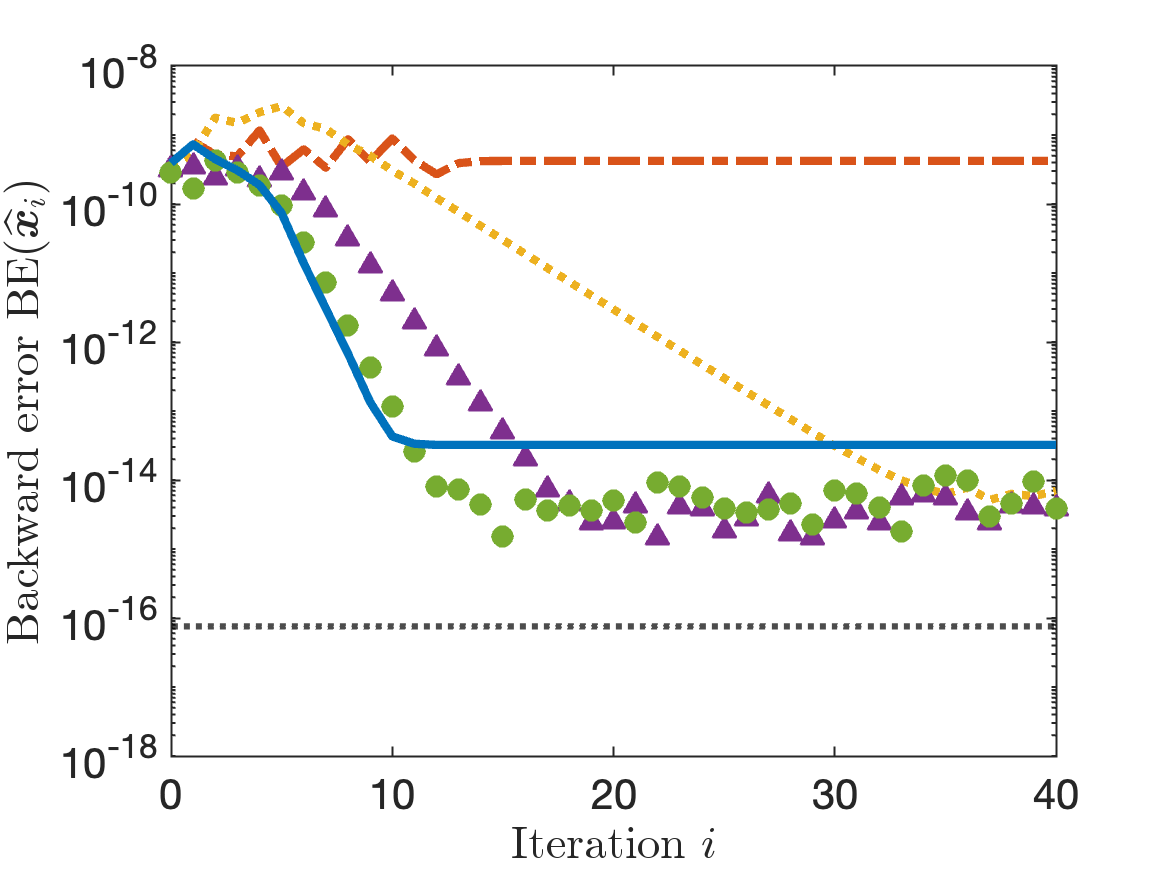} 
    
  \caption{Forward (\emph{left}), residual (\emph{center}), and backward (\emph{right}) error for sketch-and-precondition (\textbf{S{\&}P}) and iterative sketching (\textbf{IS}) with $\kappa = 10^{10}$ and $\norm{\vec{r}(\vec{x})} = 10^{-6}$.
    We consider sketch-and-precondition with both the zero initialization (red dashed) and the sketch-and-solve initialization (blue solid).
    We consider the basic iterative sketching method (yellow dotted) along with variants incorporating damping (purple circles) and momentum (green triangles)
    Reference errors for Householder \QR are shown as black dotted lines.} \label{fig:other_full}
\end{figure}

\Cref{fig:other_full} reproduces \cref{fig:other}, adding iterative sketching with damping and momentum.
Recall that in this figure, all methods are implemented with $d = 20n$.
Similar to the basic method, iterative sketching with damping and momentum appear to be forward---but not backward---stable.
Iterative sketching with damping and momentum converge at a $1.5\times$ and $2.2\times$ faster rate than the basic method, respectively.
In particular, iterative sketching with momentum converges at the same rate as sketch-and-precondition.

\begin{figure}[t]
  \centering
  
  \includegraphics[width=0.48\textwidth]{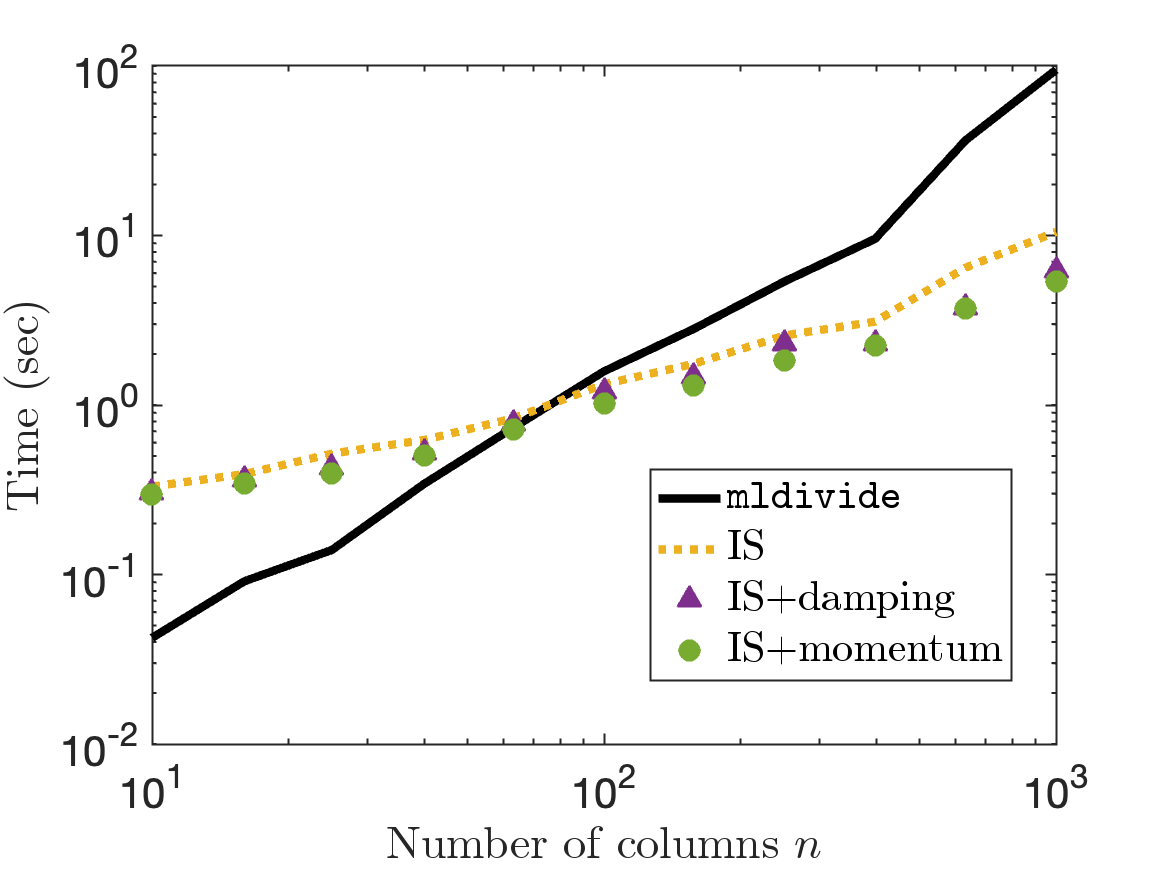}
  \hfill
  \includegraphics[width=0.48\textwidth]{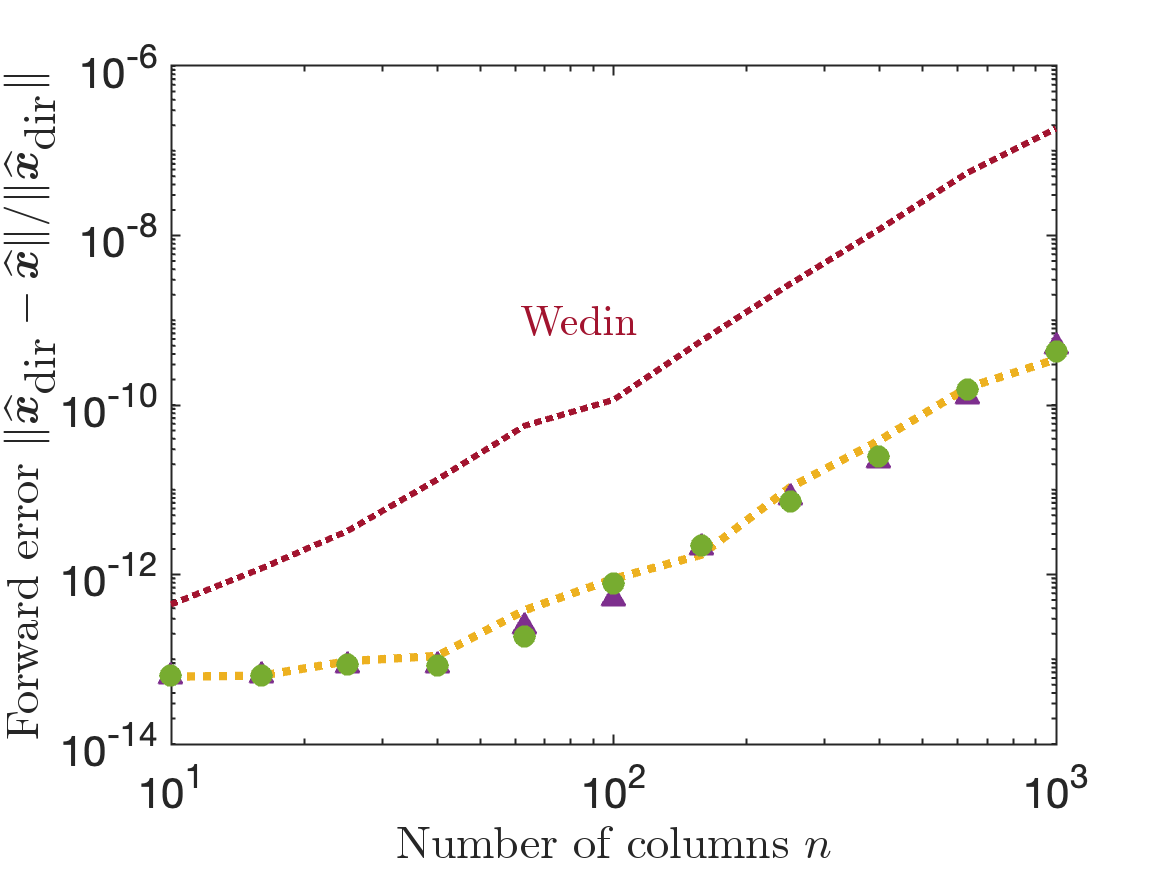}
    
  \caption{Runtime (\emph{left}) and forward error (\emph{right}) for solution of kernel regression problem by \QR (black solid), iterative sketching (yellow dotted), iterative sketching with damping (purple circles), and iterative sketching with momentum (green triangles).} \label{fig:susy_full}
\end{figure}

Similarly, \cref{fig:susy_full} reproduces \cref{fig:susy} with iterative sketching with damping and momentum included.
The accuracies of iterative sketching with damping and momentum are comparable to the basic method, and they are faster ($1.8\times$ and $2.2\times$ faster than the basic method for $n=1000$, respectively).
In particular, \textbf{iterative sketching with momentum is 21$\times$ faster than \texttt{mldivide} for this problem.}


\ifsiam
\section{Proof of \texorpdfstring{\cref{thm:damp_mom}}{Theorem B.1}}
 \hfill $\proofbox$
\fi

\section{Additional figure}

A version of \cref{fig:numerical} with the residual error is shown in \cref{fig:numerical_full}.

\begin{figure}[t]
  \centering
  
  $\norm{\vec{r}(\vec{x})} = 10^{-12}$
  
  \includegraphics[width=0.32\textwidth]{r12_forward}
  \hfill
  \includegraphics[width=0.32\textwidth]{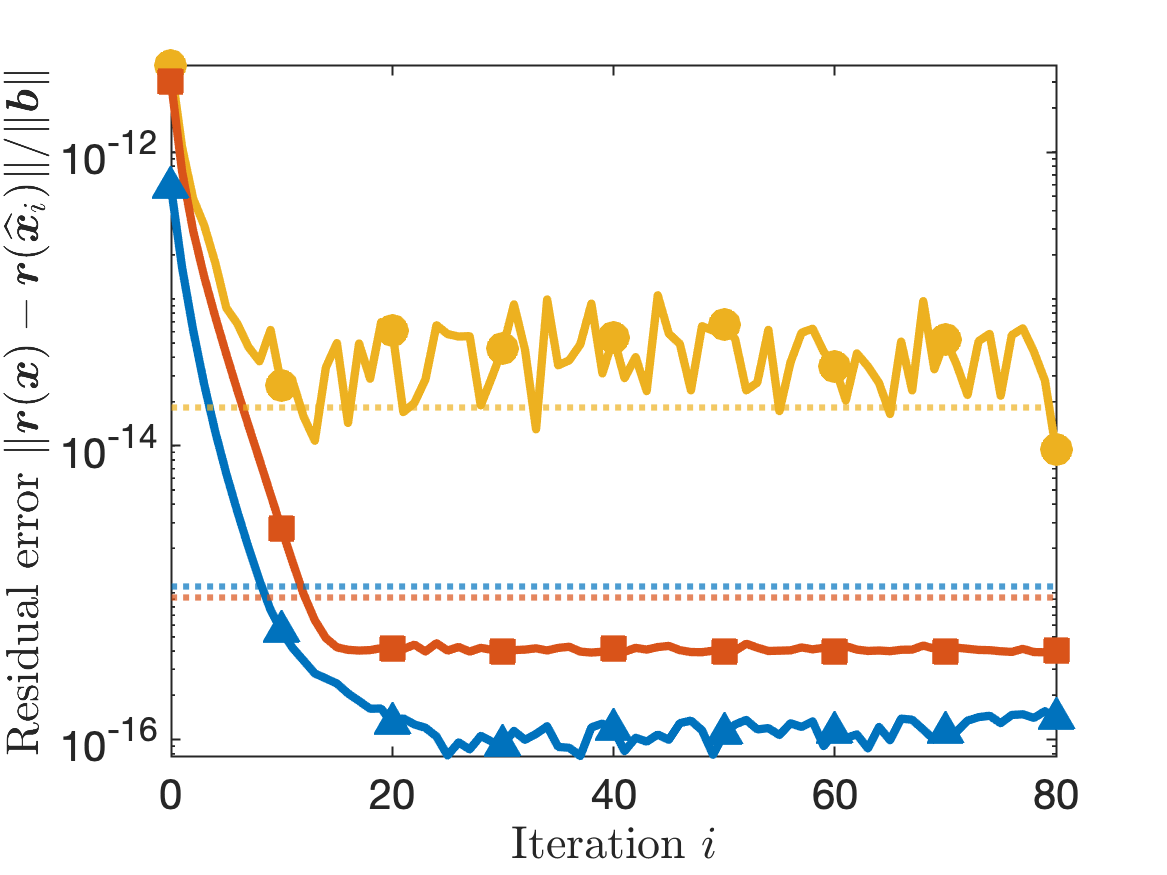}
  \hfill
  \includegraphics[width=0.32\textwidth]{r12_backward} 
  
  $\norm{\vec{r}(\vec{x})} = 10^{-3}$
  
  \includegraphics[width=0.32\textwidth]{r3_forward}
  \hfill
  \includegraphics[width=0.32\textwidth]{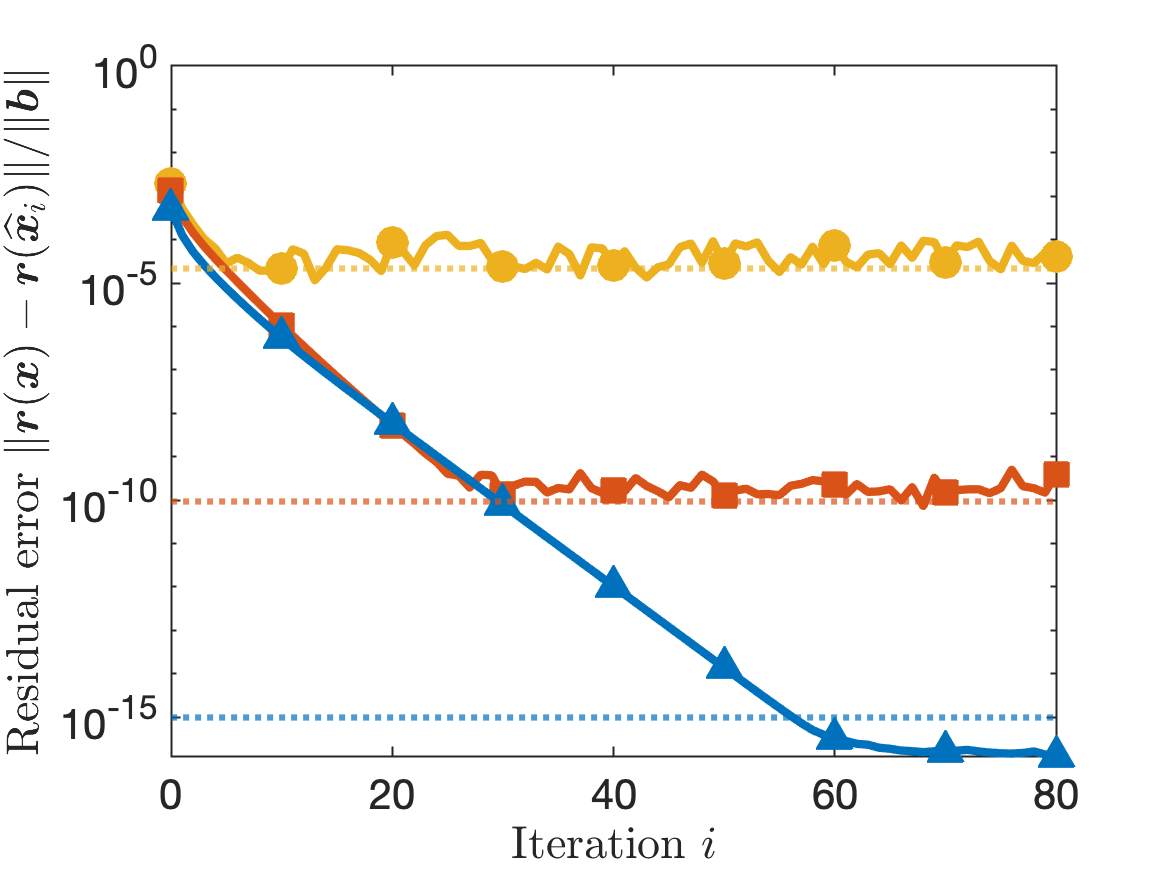}
  \hfill
  \includegraphics[width=0.32\textwidth]{r3_backward} 
  
  \caption{Forward (\emph{left}), residual (\emph{center}), and backward (\emph{right}) error for iterative sketching (solid lines) for condition numbers $\kappa = 10^1$ (blue triangles), $\kappa = 10^{10}$ (red squares), and $\kappa = 10^{15}$ (yellow circles) with two residuals $\norm{\vec{r}(\vec{x})} = 10^{-12}$ (\emph{top}) and $\norm{\vec{r}(\vec{x})} = 10^{-3}$ (\emph{bottom}).
    Reference accuracies for Householder \QR are shown as dotted lines.
    This is a version of \cref{fig:numerical} which adds an additional panel for the residual error.} \label{fig:numerical_full} 
\end{figure}

\end{document}